\documentclass[10pt,reqno]{amsart}
\usepackage[centertags]{amsmath}
\usepackage{pstricks}
\usepackage{tikz}
\usetikzlibrary{decorations.markings}
\usepackage{setspace}
\usepackage{calligra}
\usepackage[T1]{fontenc}
\usepackage{mathrsfs} 
\usepackage{amsfonts}
\usepackage[top=1.5in, bottom=1.5in, left=1.5in, right=1.5in]{geometry}
\usepackage[utf8]{inputenc}
\usepackage{graphicx}
\usepackage{sidecap}

\theoremstyle{plain}
\newtheorem{theo}{Theorem}[section]
\newtheorem{lemma}[theo]{Lemma}
\newtheorem{cor}[theo]{Corollary}
\newtheorem*{lemma*}{Lemma}

\newtheorem*{ax*}{Axiom}

\theoremstyle{remark}
\newtheorem{remark}[theo]{Remark}
\newtheorem{definition}[theo]{Definition}

\newcommand{\W}{W(x_{1},\dots,x_{N})}
\newcommand{\Proj}{\mathbb{P}(w_{1},\dots,w_{N})}

\newcommand{\R}{\mathbb{R}}

\newcommand{\C}{\mathbb{C}}
\newcommand{\K}{\kappa_{\mathcal{X}}}
\newcommand{\IsoTo}{\xrightarrow{\,\smash{\raisebox{-0.65ex}{\ensuremath{\scriptstyle\sim}}}\,}}
\newcommand{\w}{\omega_{\text{log}}}
\newcommand{\mult}{\text{age}_\sigma(\mathcal{L})}
\newcommand{\Wmod}{\text{\calligra{W}}_{0,n}^{\quad d}(k_1,\dots, k_n)}
\newcommand{\Wext}{\widetilde{\text{\calligra{L  }}^{\text{  }}}}
\newcommand{\Luni}{\text{\calligra{L}}}
\newcommand{\Ccoarse}{\left| \text{\calligra{C  }} \text{  }\right |}

\newcommand{\parlengths}{\setlength{\parindent}{0pt}}
\begin{document}

\title{Asymptotic Expansion and  the LG/(Fano, General Type) Correspondence}
\author{Pedro Acosta}
\address{Department of Mathematics, The University of Michigan, Ann Arbor, MI 48109, USA}
\email{peacosta@umich.edu}



\begin{abstract}
The celebrated LG/CY correspondence asserts that the Gromov-Witten theory of a Calabi-Yau (CY) hypersurface in weighted projective space is equivalent to its corresponding FJRW-theory
(LG) via analytic continuation. It is well known that this correspondence fails in non-Calabi-Yau cases. The main obstruction is a collapsing or dimensional reduction of the state space of the Landau-Ginzburg model in the Fano case, and a similar collapsing of the state space of Gromov-Witten theory in the general type case. We state and prove a modified version of the cohomological correspondence that describes this collapsing phenomenon at the level of state spaces. This result confirms a physical conjecture of Witten-Hori-Vafa. The main purpose of this article is to provide a quantum explanation for the collapsing phenomenon. A  key observation is that the corresponding Picard-Fuchs equation develops irregular singularities precisely at the points where the collapsing occurs. Our main idea is to replace analytic continuation with asymptotic expansion in this non-Calabi-Yau setting. The main result of this article is that the reduction in rank of the Gromov-Witten $I$-function due to power series asymptotic expansions matches precisely the dimensional reduction of the corresponding state space. Furthermore, asymptotic expansion
under a different asymptotic sequence yields a different $I$-function which can be considered as the mathematical counterpart to the additional "massive vacua" of physics.
\end{abstract}
\maketitle \parlengths

\begingroup
\baselineskip=1pt
\tableofcontents
\endgroup



\section{Introduction}\label{introduction}

The aim of the present work is to study a generalization of the famous Landau-Ginzburg/Calabi-Yau (LG/CY) correspondence to non-Calabi-Yau quotient stacks inside weighted projective spaces. 
The LG/CY correspondence relates the FJRW theory of a quasi-homogeneous polynomial $\W$ of degree $d$ with integer weights $w_{1},\dots, w_{N}$ satisfying $\sum_{i}w_{i}=d$ (CY-condition), and the Gromov-Witten theory of the degree $d$ Calabi-Yau hypersurface $\mathcal{X}_{W}=\{W=0\}$ inside of weighted projective space $\mathbb{P}(w_{1},\dots,w_{N})$. In this paper, we are mainly interested in the non-Calabi-Yau setting $\sum_{i}w_{i}\neq d$. Under this condition, $\mathcal{X}_{W}$ is a Fano hypersurface if $\sum_{i}w_{i}>d$, and a general type hypersurface if $\sum_{i}w_{i} < d$. As observed in \cite{Witten, Hori-Vafa}, in the non-Calabi-Yau setting there is a certain reduction in the dimension of the state spaces (which shall be referred to as \textit{collapsing phenomenon}). More precisely, in the Fano case, when we move from the large complex radius limit to the Landau-Ginzburg point, the dimension of the state space will decrease due to the appearance of certain "massive vacua." In other words, the Gromov-Witten state space will degenerate into a corresponding FJRW state space of smaller dimension. In the case of a general type hypersurface, the role of the Gromov-Witten and FJRW state spaces will be reversed.  In a private communication, Hori explained to us a precise conjecture that allowed us to establish a correspondence at the level of state spaces despite this collapsing phenomenon. In the first part of this paper we verify Hori's conjecture.
 
In the second part of this article, we establish a correspondence at the level of genus zero theories. A great deal of work has already been done in the Calabi-Yau setting. See, for example,  \cite{Chiodo-Ruan2} for the quintic three-fold, \cite{Priddis-Shoemaker} for the mirror quintic, \cite{CIR} for general Calabi-Yau hypersurfaces in weighted projective space, and \cite{Clader} for examples of Calabi-Yau complete intersections. The general strategy in the CY setting is to relate the $I$-functions of the two theories via analytic continuation. The non-CY case, however, presents a seemingly insurmountable obstacle: the Picard-Fuchs operator of the theory develops an irregular singularity at the Landau-Ginzburg point in the Fano case, and at the large complex radius limit in the general type case, making analytic continuation of the $I$-functions impossible. To add to the problem, the $I$-function of the theory at the irregular singularity radius of convergence equal to zero, rendering the function a formal power series. To overcome these obstacles, we appeal to the theory of asymptotic expansions. We show that in this setting, the $I$-functions of the two theories are naturally related via power series asymptotic expansion.
\subsection{Statement of Results}
In order to state the cohomological correspondence, we introduce some notation. Let $\W$ be a \textit{quasi-homogeneous} polynomial of degree $d$ and integer weights $w_{1},\dots, w_{N}$, i.e. the weights of $W$ satisfy
\begin{equation}\label{quasi-homogenous}
W(\lambda^{w_{1}}x_{1},\dots,\lambda^{w_{N}}x_{N})=\lambda^d\W
\end{equation}
for all $\lambda\in \C^{\ast}$. We assume throughout that $\text{gcd}(w_{1},\dots,w_{N})=1$ and that $W$ has a unique singularity at the origin. Then, $\mathcal{X}_{W}:=\{W=0\}$ defines a smooth hypersurface inside of the weighted projective stack $\mathbb{P}(w_{1},\dots,w_{N})$, i.e. $\mathcal{X}_{W}$ is a Deligne-Mumford stack. Define the \textit{canonical bundle index} of $\mathcal{X}_W$ to be
\begin{equation}
\K:=d-\sum_{j=1}^N w_j.
\end{equation}
When $\mathcal{X}_W$ is Fano, we define its \textit{Fano index} to be $r:=-\K$. This matches the notation of \cite{Iritani}.

Let $\text{Aut}(W)\subset (\C^{\ast})^N$ be the group of diagonal symmetries that leaves $W$ invariant under rescaling of the coordinates:
\begin{equation}
W(\alpha_{1}x_{1},\dots,\alpha_{N}x_{N})=W(x_{1},\dots,x_{N}),
\end{equation}
for all $(x_{1},\dots,x_{N})\in \C^N$. Because of the quasi-homogeneous condition (\ref{quasi-homogenous}), the group $\text{Aut}(W)$ always contains the element 
\begin{equation}\label{J}
J_{W}:=(\text{exp}(2\pi i w_{1}/d),\dots,\text{exp}(2\pi i w_{N}/d)).
\end{equation}
When there is no possibility of confusion, we denote $J_W$ simply by $J$ to simplify the notation. Let $G$ be a subgroup of $\text{Aut}(W)$ containing $J_{W}$ and define $\widetilde{G}:=G/ \langle J_{W}\rangle$.

Chiodo and Ruan showed in \cite{Chiodo-Ruan} that if the weights of $W$ satisfy the Calabi-Yau condition $\sum_{i=1}^{N}w_{i} = d$, then the Chen-Ruan cohomology of the Calabi-Yau quotient stack $[\mathcal{X}_{W}/\widetilde{G}]$ is isomorphic to the FJRW state space of the pair $(W, G)$, i.e. we have a bi-degree preserving isomorphism
\[
H_{\text{CR}}^{p,q}([\mathcal{X}_{W}/\widetilde{G}];\C)\cong H_{FJRW}^{p,q}(W,G).
\]

(Detailed definitions of these spaces will be given section \ref{correspondence}.)

In the non-Calabi-Yau case, however, such an isomorphism is impossible since the dimensions of the states spaces are different. To rectify this phenomenon, we need to add certain additional classes to FJRW theory in the Fano case, and to Gromov-Witten theory in the general type case.
Even with these additions,  the isomorphism of state spaces preserves only a difference of gradings and not a bi-grading like in the Calabi-Yau case. To state the results of the paper in more detail, we distinguish between the following cases:

\textbf{Fano case:} $\K<0$.\\
Define the \textit{modified Chen-Ruan state space} to be
\begin{equation}\label{modified-Chen-Ruan-state-space-Fano}
\mathcal{A}_{\text{CR}}^{n}([\mathcal{X}_{W}/\widetilde{G}];\C):=\bigoplus_{p-q=n}H_{\text{CR}}^{p,q}([\mathcal{X}_{W}/\widetilde{G}];\C)
\end{equation}
for all $n\in \mathbb{Z}$, and the \textit{modified FJRW state space} to be
\begin{equation}\label{modified-FJRW-state-space-Fano} 
 \mathcal{A}_{\text{FJRW}}^{n}(W,G) :=
  \begin{cases}
    \bigoplus_{p}H_{\text{FJRW}}^{p,p}(W,G) \oplus\C^{\frac{|G|}{d}r} & \text{if } n=0 \\
	\\
   \bigoplus_{p-q=n}H_{\text{FJRW}}^{p,q}(W,G)       & \text{otherwise.}
  \end{cases}
\end{equation}
Note that since $\K<0$ and $|\widetilde{G}|=|G|/d$, we have that $\frac{|G|}{d}r$ is a positive integer.

\textbf{General type case:}   $\K>0$.\\
In this case, the modified state spaces are defined to be
\begin{equation}\label{modified-Chen-Ruan-state-space-gt}
\mathcal{A}_{\text{CR}}^{n}([\mathcal{X}_{W}/\widetilde{G}];\C):=
\begin{cases}
    \bigoplus_{p}H_{\text{CR}}^{p,p}([\mathcal{X}_{W}/\widetilde{G}];\C) \oplus\C^{\frac{|G|}{d}\K} & \text{if } n=0 \\
	\\
   \bigoplus_{p-q=n}H_{\text{CR}}^{p,q}([\mathcal{X}_{W}/\widetilde{G}];\C)       & \text{otherwise}
  \end{cases}
\end{equation}
and
\\
\begin{equation}\label{modified-FJRW-state-space-gt} 
 \mathcal{A}_{\text{FJRW}}^{n}(W,G) := \bigoplus_{p-q=n}H_{\text{FJRW}}^{p,q}(W,G) 
\end{equation}
for all $n\in \mathbb{Z}$.

The modified cohomological correspondence asserts that we have a degree preserving isomorphism between the modified state spaces described above:
\begin{theo}[Cohomological Correspondence]\label{cohomological-correspondence}
Let $\W$ be a non-degenerate quasi-homogeneous polynomial of degree $d$ and integer weights $w_{1},\dots, w_{N}$, satisfying the condition $\K\neq0$. Let $G$ be a group of diagonal symmetries of $W$ containing the element $J_{W}$ defined in $(\ref{J})$. Then, we have a graded isomorphism of $\C$-vector spaces
\begin{equation}\label{correspondence-equation}
\mathcal{A}_{\text{CR}}^{n}([\mathcal{X}_{W}/\widetilde{G}];\C) \cong   \mathcal{A}_{\text{FJRW}}^{n}(W,G) \quad\text{for all $n\in\mathbb{Z}$,}
\end{equation} 

where $[\mathcal{X}_{W}/\widetilde{G}]$ is the Deligne-Mumford stack defined in $(\ref{stack})$. This isomorphism is canonical for $n\neq 0$.
\end{theo}
The above theorem confirms the conjecture communicated to us by Hori. It is important to remark that to establish the correspondence at the level of state spaces, the Gorenstein condition is not necessary. Example \ref{non-gorenstein} illustrates the cohomological correspondence for a non-Gorenstein case.

In the second half of this paper we state a quantum correspondence, with three important restrictions. First, we assume that the group $G$ is generated by $J_{W}$, i.e. $G=\langle J_{W}\rangle$. We also assume that the hypersurface $\mathcal{X}_W$ is Gorenstein. This is equivalent to the condition $w_j | d$ for $j=1,\dots,N$. Lastly, we restrict to ambient cohomology classes in Gromov-Witten theory and to narrow sectors in FJRW theory (see Definition \ref{narrow-sectors}.) 

The genus zero part of both, Gromov-Witten theory and FJRW theory, is completely determined by their \textit{Givental $J$-functions}. The definition of the $J$-function in Gromov-Witten theory is well known (see, for example, \cite{Givental}). For the case at hand, the $J$-function can be computed via the Givental mirror theorem using the Gromov-Witten $I$-function of $\mathcal{X}_W$:
\begin{equation}
I_{\text{GW}}^{small}(q,z):=zq^{P/z}\sum_{\substack{n\in\mathbb{Q}_{\geq 0}\\ \exists j:\text{ } nw_j\in \mathbb{Z}}} q^n\frac{\prod_{\substack{0<b\leq nd \\ \langle b\rangle=\langle nd\rangle}} (dP+bz)}{\prod_{j=1}^N\prod_{\substack{0<b\leq nw_j \\ \langle b\rangle=\langle nw_j\rangle}} (w_jP+bz)}\textbf{1}_{\langle -n\rangle},
\end{equation}
which was computed in \cite{CCLT}. Here, $\langle b\rangle:=b-\lfloor b\rfloor$ is the fractional part of $b$. This $I$-function represents a complete set of solutions to the irreducible component of the following Picard-Fuchs equation:
\begin{equation}\label{PF-equation}
\left [\prod_{j=1}^N\prod_{c=0}^{w_j-1}(w_jzD_q-cz)-q\prod_{c=1}^d(dzD_q+cz)\right ]I_{\text{GW}}^{small}(q,z)=0,\quad\text{where $D_q:=q\frac{d}{dq}$.}
\end{equation}
A simple analysis reveals that if $\K<0$, this differential operator has an irregular singularity at $q=\infty$, and that if $\K>0$, the operator has an irregular singular point at $q=0$, rendering $I_{\text{GW}}^{small}(q,z)$ a formal solution at $q=0$. We are thus presented with a dichotomy similar to the one observed in the correspondence of state spaces. 

On the Landau-Ginzburg side, let $\{\phi_0,\dots,\phi_s\}$ be a basis for the narrow part of FJRW theory (see Definition \ref{narrow-sectors}) and $\{\phi^0,\dots, \phi^s\}$ its dual basis with respect to the FJRW pairing defined in Equation (\ref{grading}). 
The \textit{big FJRW $J$-function} is defined to be
\begin{equation}
J_{\text{FJRW}}\left(\textbf{t}=\sum_{i=0}^s t_0^i\phi_i,z\right):=z\phi_0+\sum_{h} t_0^h\phi_h+\sum_{\substack{n\geq 0\\ k\geq 0}}\text{ }\sum_{h_1,\dots,h_n}\sum_{\epsilon}\frac{t_{0}^{h_1}\dots t_{0}^{h_n}}{n!z^{k+1}} \langle \phi_{h_1},\dots, \phi_{h_n},\tau_k(\phi_{\epsilon})\rangle_{0,n+1}^{\text{FJRW}} \phi^{\epsilon},
\end{equation}
where $\langle \phi_{h_1},\dots, \phi_{h_n},\tau_k(\phi_{\epsilon})\rangle_{0,n+1}^{\text{FJRW}}$ are the genus zero invariants of FJRW theory, defined in Equation (\ref{invariants}). The \textit{big $I$-function} for the narrow part of FJRW theory is given by
\begin{equation}
I_{\text{FJRW}}(\textbf{t},z):=z\sum_{ \substack{k_i\geq 0:i\in\textbf{Nar} \\ \widetilde{h}(\textbf{k})\in\textbf{Nar} }}\prod_{i}\frac{(t_0^{i})^{k_i}}{z^{k_i}k_i!} \left(\prod_{j=1}^N \prod_{\substack{0 \leq b <q_j\sum_i ik_i\\ \langle b\rangle = \langle q_j\sum_i ik_i\rangle}} (q_j +b)z \right)\phi_{\widetilde{h}(\textbf{k})},
\end{equation}
where $\widetilde{h}(\textbf{k}):=d\left\langle\frac{\sum_i ik_i}{d}\right\rangle$ and the set $\textbf{Nar}$ is defined in Equation (\ref{narrow-set}). The FJRW $I$-function can be related to the FJRW $J$-function by the following Coates-Givental-style mirror theorem:
\begin{theo}[FJRW Mirror Theorem]\label{FJRW-mirror} The FJRW $J$-function can be uniquely determined from the FJRW $I$-function and its derivatives. To be more explicit, we have the following relation:
\begin{equation}\label{mirror-equation}
J_{\text{FJRW}}(\boldsymbol{\tau},-z)=I_{\text{FJRW}}(\textbf{t},-z)+\sum_{i\in\textbf{Nar}} c_i(t,z)z\frac{\partial I_{\text{FJRW}}}{\partial t_0^i}(\textbf{t},-z),
\end{equation}
where $c_i(\textbf{t},z)$ is a formal power series in $\textbf{t}$ and $z$, for $i\in\textbf{Nar}$, and $\boldsymbol{\tau}(\textbf{t})$ is determined by the $z^0$-mode of the right-hand-side of Equation $(\ref{mirror-equation})$.
\end{theo}

By restricting the FJRW $I$-function to the slice $t_0^i=0$ for $i\neq 1$, we obtain the \textit{small FJRW $I$-function}, which is defined as follows:

\begin{equation}
I_{\text{FJRW}}^{small}(t,z):=z\sum_{k\in \textbf{Nar}}\sum_{l=0}^\infty \frac{t^{dl+k+1}}{z^{dl+k}(dl+k)!}\frac{\prod_{j=1}^N z^{\lfloor q_j(dl+k) \rfloor}\Gamma(q_j(dl+k+1))}{\Gamma(q_j+\langle q_jk\rangle)}  \phi_k.
\end{equation}

In the Calabi-Yau case, $I_{\text{FJRW}}^{small}(t,z)$ and $I_{\text{GW}}^{small}(q,z)$ have the same rank (as a consequence of state space isomorphism) and are related by analytic continuation.
In the non-Calabi-Yau case,  $I_{\text{FJRW}}^{small}(t,z)$ and $I_{\text{GW}}^{small}(q,z)$ have different ranks as cohomology-valued functions. Consequently, it is not precisely clear how to relate them. In terms of the Picard-Fuchs equation, we have two diametrically opposing cases which depend on the sign of the canonical bundle index $\K$. If $\K>0$, $I_{\text{FJRW}}^{small}(t,-z)$ converges for all $t$ and it represents a complete set of solutions to the irreducible component of the Picard-Fuchs operator of Equation (\ref{PF-equation}) at $q=\infty$, where we have used the change of variables $q=t^{-d}$. On the other hand, if $\K<0$, this series diverges for $t\neq 0$ and it represents a formal solution to the Picard-Fuchs operator at $q=\infty$. Therefore, we conclude that the genus zero correspondence must be divided into two cases, which we describe now.

\textbf{Fano case:} $\K<0$.\\
As explained above, in this case the Picard-Fuchs operator has an irregular singular point at $q=\infty$. A simple anlysis also shows that $I_{\text{GW}}^{small}(q,1)$ converges absolutely for all $q\neq \infty$ and it develops an essential singularity at $q=\infty$. As a consequence, we are not able to analytically continue this function to $q=\infty$ as in the Calabi-Yau case. We propose to replace analytic continuation with power series asymptotic expansion. At the same time, $I_{\text{FJRW}}^{small}(t,-1)$ has smaller rank than $I_{\text{GW}}^{small}(q,1)$ as cohomology-valued functions. Intuitively, our theorem asserts that the asymptotic expansion of $I_{\text{GW}}^{small}(q,z)$  collapses its rank to match the rank of $I_{\text{FJRW}}^{small}(t,-1)$.
More precisely, 

\begin{theo}[Genus zero LG/Fano Correspondence]\label{fano-genus-zero-correspondence}
If $\K<0$, there exists a unique linear transformation $L_{\text{GW}}:H_{\text{CR}}^{amb}(\mathcal{X}_W;\C) \longrightarrow H_{\text{FJRW}}^{nar}\left(W\right)$ of rank equal to the dimension of $ H_{\text{FJRW}}^{nar}\left(W\right)$ such that
\[
L_{\text{GW}}\cdot I_{\text{GW}}^{small}(q,1)\sim I_{\text{FJRW}}^{small}(t=q^{-\frac{1}{d}},-1)
\]
as $q\rightarrow\infty$ from some suitable sector of the complex plane.
\end{theo}
This theorem yields the following important consequence:
\begin{cor}\label{corollary1.4}
If $\K<0$, the genus zero FJRW invariants are completely determined by the genus zero Gromov-Witten invariants of $\mathcal{X}_W$.
\end{cor}
\textbf{General type case:} $\K>0$.\\
In the general type case the story is exactly the opposite: $I_{\text{GW}}^{small}(q,1)$ diverges for $q\neq 0$ and $I_{\text{FJRW}}^{small}(t,-1)$ converges for $t\neq\infty$. $I_{\text{FJRW}}^{small}(t,-1)$ has larger rank than $I_{\text{GW}}^{small}(q,1)$ as cohomology-valued functions. As in the Fano case, it is not possible to relate the two series by analytic continuation. However, we have the following result in terms of power series asymptotic expansions:
\begin{theo}[Genus zero LG/General type Correspondence]\label{general-type-genus-zero-correspondence}
If $\K<0$, there exists a unique linear transformation $L_{\text{FJRW}}:H_{\text{FJRW}}^{nar}\left(W\right)\longrightarrow L_{\text{FJRW}}$ of rank equal to the dimension of $H_{\text{CR}}^{amb}(\mathcal{X}_W;\C)$ such that
\[
L_{\text{FJRW}}\cdot I_{\text{FJRW}}^{small}(t,-1)\sim I_{\text{GW}}^{small}(q=t^{-d},1)
\]
as $t\rightarrow\infty$ from some suitable sector of the complex plane.
\end{theo}
The LG/General type correspondence implies the following important result:
\begin{cor}\label{corollary1.6}
If $\K>0$, the genus zero GW invariants of $\mathcal{X}_W$ are completely determined by the genus zero FJRW invariants of the pair $\left(W,\langle J_W\rangle\right)$.
\end{cor}
Lastly, we observed that the information corresponding to the "massive vacua" is encoded in formal solutions to the Picard-Fuchs equation with an exponential factor as the leading term. For a degree $d$ Fano hypersurface in $\mathbb{P}(w_1,\dots,w_N)$, the \textit{massive vacuum solutions} are defined to be
\[
I_{j,mass}(q):=q^{-\frac{N-2}{2r}} e^{\alpha_j q^{\frac{1}{r}} }\sum_{n=0}^{\infty}\frac{a_{j,n}}{q^{n/r}}\quad\text{for $j=1,\dots,r,$}
\]
where $r=\sum_j w_j-d$, and $\{\alpha_j\}_{j=1}^r$ are the distinct roots of $\left(\frac{\alpha}{r}\right)^r=d^d\prod_{j=1}^N w_j^{-w_j}$. These formal series are solutions to the Picard-Fuchs equation at $q=\infty$. If we restrict to projective space $\mathbb{P}^{N-1}$, we can relate the Gromov-Witten $I$-function to the massive vacuum solutions via asymptotic expansion:
\begin{theo}\label{intro-massive-theorem}
Let $I_{\text{GW}}(q,1)$ be the Gromov-Witten $I$-function of a degree $d$ hypersurface inside $\mathbb{P}^{N-1}$. Then,
\[
I_{\text{GW}}(q=u^r,1)\sim C' \frac{\Gamma(1+P)^N}{\Gamma(1+dP)} u^{-\frac{(N-2)}{2}} e^{\alpha u} (1+\mathcal{O}\left(u^{-1} )\right)\quad\text{as $u\rightarrow +\infty$ along the real axis,}
\]
where $C'$ is a constant, $r=N-d$, and $\alpha>0$ satisfies $\left(\frac{\alpha}{r}\right)^r=d^d$.
\end{theo}

A more general version of this theorem, which applies to hypersurfaces inside Fano manifolds, was first proved by Galkin-Golyshev-Iritani (see \cite[Section 5.3]{Iritani}) in the context of the Gamma conjectures.

\subsection{Plan of the paper}

The paper is organized as follows. In Section \ref{correspondence} we define Chen-Ruan cohomology and the FJRW state space, prove the cohomological correspondence, and present explicit examples of the correspondence. Section \ref{quantum-theory} contains a review of the genus zero FJRW theory. We also review the formalism of Givental and construct the FJRW $I$-function. In Section \ref{correspondence-asymptotic} we define asymptotic expansions and provide a proof of the genus zero correspondence and its corollaries.
\subsection{Acknowledgments}
I am greatly indebted to my adviser Yongbin Ruan for introducing me to this problem and for the countless hours of advice I have received from him. I would also like to thank Kentaro Hori for explaining his work to me during his visit to Michigan and for patiently listening to my proof of the cohomological correspondence. This paper owes greatly to his work. I would like to thank Peter Miller. I received great help from him while I was studying the asymptotic analysis of the $I$-function. The examples found in his book on asymptotic analysis were particularly useful. I would also like to acknowledge Hiroshi Iritani for his help in the process of writing this paper. In many ways he has acted as a second adviser. Lastly, I would like to acknowledge Alessandro Chiodo, Emily Clader, Jeremy Hoskins, Tyler Jarvis, Albrecht Klemm, Dustin Ross, and Mark Shoemaker for expressing interest in my work, for their valuable advice, and for many fruitful discussions.




\section{The Modified Cohomological Correspondence}\label{correspondence}

In this section we define Chen-Ruan cohomology and the state space of FJRW theory. We mainly follow the presentation found in \cite{Chiodo-Ruan}.

\subsection{The set-up}

We consider quasi-homogenous polynomials of the form
\begin{equation}\label{polynomial}
\W=c_1\prod_{j=1}^{N}x_{j}^{b_{1,j}} + \dots + c_s\prod_{j=1}^{N} x_{j}^{b_{s,j}}
\end{equation}
where $c_{i}\neq 0$ and $b_{i,j}\in \mathbb{Z}_{\geq 0}$. We further assume that each monomial is distinct. A quasi-homogenous polynomial of degree $d$ and integer weights $w_{1},\dots, w_{N}$ is said to be \textit{non-degenerate} if it satisfies the following conditions:

\begin{enumerate}
\item The polynomial $W$ has a unique singular point at the origin;\hfill\refstepcounter{equation}(\theequation)\label{non-degenerate}
\item the rational numbers $q_{i}:=w_i/d$ (known as \textit{charges}) are  completely determined by the polynomial $W$.\hfill\refstepcounter{equation}(\theequation)
\end{enumerate}
\subsubsection{ The state space of Gromov-Witten theory} We use Chen-Ruan cohomology as the state space of Gromov-Witten theory. We briefly review its construction for the case of hypersurfaces in weighted projective space.

Consider the weighted projective stack
\[
\Proj =[ \C^{N}\backslash\{\mathbf{0}\}/\C^{*}]
\]
where $\C^{\ast}$ acts on $\C^{N}\backslash\{\mathbf{0}\}$ by $\lambda\cdot (x_{1},\dots,x_{N})=(\lambda^{w_{1}}x_{1},\dots,\lambda^{w_{N}}x_{N})$.

Inside $\C^{N}\backslash\{\mathbf{0}\}$, the locus $\{W=0\}$ defines a smooth hypersurface since $(\partial_{i}W(x))_{i=1}^N\neq 0$ on $\C^{N}\backslash\{\mathbf{0}\}$ (by the non-degeneracy condition (\ref{non-degenerate})). Because the action of $\C^{\ast}$ fixes $\{W=0\}$, we can define a quotient stack $\mathcal{X}_{W}$ by
\begin{equation}
\mathcal{X}_{W}:=[\{W=0 \}_{\C^{N}\backslash\{\mathbf{0}\}}/\C^{\ast}] \subset \Proj
\end{equation}

Consider now a subgroup $G$ of $\text{Aut}(W)$ containing the element $J_{W}$. The group $\widetilde{G}:=G/\langle J_{W}\rangle$ acts faithfully on the stack $\mathcal{X}_{W}$ and so, we can define a quotient stack $[\mathcal{X}_{W}/\widetilde{G}]$. 

Define a homomorphism of groups from $\C^{\ast}$ to $(\C^{\ast})^N$ by
\begin{equation}
\lambda \mapsto \bar{\lambda}:=(\lambda^{w_1},\dots, \lambda^{w_N}).
\end{equation}
This map is an injection since $\bigcap_{i=1}^N \boldsymbol{\mu}_{w_i}=\{1\}$ (this follows from $\text{gcd}(w_1,\dots,w_N)=1$). Under this map, the multiplicative group of $dth$-roots of unity $\boldsymbol{\mu}_d$ maps to $\langle J_{W}\rangle$ and $G\cap \C^{\ast}=\langle J_{W}\rangle$. It follows that $\widetilde{G}\cong G\C^{\ast}/\C^{\ast}$.

From \cite[Remark 2.4]{Romagni} and the isomorphism $\widetilde{G}\cong G\C^{\ast}/\C^{\ast}$, we have the following isomorphism of stacks
\begin{equation}\label{stack}
[\mathcal{X}_{W}/\widetilde{G}]\cong[\{W=0 \}_{\C^{N}\backslash\{\mathbf{0}\}}/G\C^{\ast}]
\end{equation}

where $G\C^{\ast}:=\{g(\lambda^{w_{1}},\dots,\lambda^{w_{N}}) \mid g\in G \subset (\C^\ast)^{N}, \lambda\in \C^{\ast}  \}$. We will use Equation (\ref{stack}) as the working definition of $[\mathcal{X}_{W}/\widetilde{G}]$. 

We want to compute the Chen-Ruan cohomology with complex coefficients of this stack. We briefly review this construction for a smooth Deligne-Mumford stack of the form $\mathcal{X}=[\mathcal{U}/H]$, where $H$ is an Abelian group (for a more comprehensive review, see \cite{Chen-Ruan, ALR}). As a $\C$-vector space, the Chen-Ruan cohomology of $\mathcal{X}$  is the cohomology with $\C$-coefficients of the \textit{inertia stack} $\mathcal{IX}:=\bigsqcup_{h\in H} \mathcal{X}_h$, where the so-called \textit{$h$-sector} is defined to be
\begin{equation}
\mathcal{X}_h:=[\{ u\in \mathcal{U} \mid h\cdot u = u\}/H]
\end{equation}

Since we are working with complex coefficients, $H^{\ast}(\mathcal{X}_{h};\C)$ is isomorphic to the cohomology of the underlying coarse quotient scheme of $\mathcal{X}_{h}$.

For the case under consideration, we define the following notation:
\begin{definition}
For $\gamma=g\bar{\lambda}=(g_1\lambda^{w_1},\dots, g_N\lambda^{w_N})\in G\C^\ast$ define
\begin{align}
	\C_{\gamma}^{N} &:=\{x\in \C^N \mid \gamma\cdot x = x\} \\
	N_{\gamma} &:=\text{dim}_{\C}(\C_{\gamma}^N) \\
	W_{\gamma} &:= W\mid_{\C_{\gamma}^N}\\
	\boldsymbol{w}_{\gamma}&:=\{1\leq j\leq\ N \mid g_j\lambda^{w_j}=1\} \\
	\{W_{\gamma}=0\}_{\gamma} &:= \{ W_{\gamma}=0\}_{\C_\gamma^N\backslash\{\mathbf{0}\}} \subset \C_\gamma^N\backslash\{\mathbf{0}\}
\end{align}
\end{definition}

The following lemma (\cite[p.12]{Chiodo-Ruan}) ensures that $\{ W_{\gamma}=0\}_\gamma$ is smooth in $\C_\gamma^N\backslash\{\mathbf{0}\}$:
\begin{lemma}\label{dichotomy}
If $\gamma\in G$, then $\{ W_{\gamma}=0\}_\gamma$ is a smooth hypersurface in $\C_\gamma^N\backslash\{\mathbf{0}\}$. If $\gamma\notin G$ then  $\{ W_{\gamma}=0\}_\gamma=\C_\gamma^N\backslash\{\mathbf{0}\}$.
\end{lemma}
\begin{remark}
The two cases observed in Lemma \ref{dichotomy} will become important in the proof of cohomological correspondence. They represent the difference between the Gorenstein and non-Gorenstein cases.
\end{remark}

The action of $\gamma\in G\C^{\ast}$ on the tangent space $\text{T}_{\boldsymbol{x}}(\{W=0\})$ of a fixed point $\boldsymbol{x}\in \{W_{\gamma}=0\}_{\gamma} $ is given, after choosing a basis, by a diagonal matrix of the form
\begin{equation}
\text{Diag}(\text{exp}(2\pi i a_{1}^{\gamma}),\dots , \text{exp}(2\pi i a_{N{-1}}^{\gamma}))
\end{equation}
where $a_{i}^{\gamma}\in [0,1)$ for all i.
\begin{definition}
The age $a_{\boldsymbol{x}}(\gamma)$ of $\gamma\in G\C^{\ast}$ acting on $\text{T}_{\boldsymbol{x}}(\{W=0\})$ is defined as
\begin{equation}\label{ageshift}
a_{\boldsymbol{x}}(\gamma):=\sum_{i=1}^{N-1}a_{i}^{\gamma}.
\end{equation}
\end{definition}
It is not hard to see that the age is independent of the basis chosen and is constant on a given sector. When no confusion arises, we denote the age simply by $a(\gamma)$.

We are finally ready to define the Chen-Ruan cohomology of the Deligne-Mumford quotient stack $[\mathcal{X}_{W}/\widetilde{G}]$:
\begin{equation}\label{CRstate}
H_{\text{CR}}^{p,q}([\mathcal{X}_{W}/\widetilde{G}];\C):=\bigoplus_{\gamma\in G\C^{\ast}} H^{p-a(\gamma),q-a(\gamma)}(\{W_{\gamma}=0\}_{\gamma}/G\C^{\ast};\C)
\end{equation}

where $a(\gamma)$ is the age defined in Equation (\ref{ageshift}).
\subsubsection{The state space of FJRW theory}

We will now construct the FJRW state space for the singularity $W:\C^{N}\rightarrow \C$ (for a more detailed construction, see \cite{FJRa}). \\
Let $\gamma$ be an element of the group of symmetries $G$. Then $\gamma$ can be expressed uniquely in the form
\[
\gamma =(\text{exp}(2\pi i \Theta_{1}^{\gamma}),\dots, \text{exp}(2\pi i \Theta_{N}^{\gamma}))
\]
with $\Theta_{j}^{\gamma}\in [0,1)$ for all $j=1,\dots,N$.

\begin{definition}
For $\gamma=(\text{exp}(2\pi i \Theta_{1}^{\gamma}),\dots, \text{exp}(2\pi i \Theta_{N}^{\gamma}))\in \text{GL}(\C,N)$, the age $a(\gamma)$ is defined as
\begin{equation}\label{age}
a(\gamma):=\sum_{j=1}^{N}\Theta_{j}^{\gamma}.
\end{equation}
\end{definition}
The corresponding FJRW $\gamma$-sector $\mathcal{H}_{\gamma}$ is defined as the $G$-invariant part of the middle dimensional relative cohomology of $\C_{\gamma}^{N}$, i.e.
\begin{equation}\label{gamma-sector}
\mathcal{H}_{\gamma}:=H^{N_{\gamma}}(\C_{\gamma}^N, W_{\gamma}^{+\infty};\C)^{G}
\end{equation}
where $W_{\gamma}^{+\infty}:=(\mathfrak{Re} W_{\gamma})^{-1}(M,+\infty)$, $M\gg 0$.

The FJRW state space is defined as the direct sum of all these sectors,
\begin{equation}\label{FJRWspace}
H_{\text{FJRW}}^{\ast}(W,G):=\bigoplus_{\gamma\in G}\mathcal{H}_{\gamma}.
\end{equation}

As in the case of Chen-Ruan cohomology, each sector is endowed with a bi-grading defined by
\begin{equation}\label{FJRWstate}
\mathcal{H}_{\gamma}^{p,q}=H^{p-a(\gamma)+\sum_{j=1}^Nq_j,q-a(\gamma)+\sum_{j=1}^Nq_{j}}(\C_{\gamma}^N, W_{\gamma}^{+\infty};\C)^{G}
\end{equation}

where $a(\gamma)$ is the age shift of $\gamma$ defined in (\ref{age}) and the $q_j$ are the charges of Equation (\ref{non-degenerate}). This bi-grading induces a bi-grading on the FJRW state space:
\begin{equation}
H_{\text{FJRW}}^{p,q}(W,G)=\bigoplus_{\gamma \in G}\mathcal{H}_{\gamma}^{p,q}.
\end{equation}

\begin{definition}\label{narrow-sectors}
A sector for which $N_{\gamma}=0$, i.e. for which the fixed locus by the action of $\gamma$ is trivial, is said to be $narrow$. If $N_{\gamma}\neq 0$ the sector is said to be $broad$.
\end{definition}
An important subspace of the FJRW state space is the space of \textit{narrow sectors} $H_{\text{FJRW}}^{nar}(W,G)$. It is defined as the direct sum of all narrow sectors of the FJRW state space. When $G=\langle J_{W}\rangle$, it is simply denoted by $H_{\text{FJRW}}^{nar}(W)$.




\subsection{A Proof of the Modified Cohomological Correspondence}\label{proof}
In this section we explicitly compute the corresponding state spaces and establish a cohomological correspondence. 
\subsubsection{Explicit computation of the Gromov-Witten state space.}

In \cite[Section 5]{Chiodo-Ruan}, Chiodo and Ruan computed the following decomposition of the Chen-Ruan cohomology  $H_{\text{CR}}^{\ast}([\mathcal{X}_{W}/\widetilde{G}];\C)$:

Since $\widetilde{G}:=G/\langle J_{W} \rangle \cong G\C^{\ast}/\C^{\ast}$, we can find $M=|G|/d$ elements $\{g^{(1)},\dots, g^{(M)}\}\subset G$ such that
\begin{equation}\label{cosets}
G\C^{\ast} = \bigsqcup_{i=1}^{M} g^{(i)}\C^{\ast}.
\end{equation}
We then have that
\begin{equation}\label{coset-decomposition}
H_{\text{CR}}^{\ast}([\mathcal{X}_{W}/\widetilde{G}];\C)=\bigoplus_{i=1}^{M}\bigoplus_{\gamma\in g^{(i)}\C^{\ast}} H^{\ast}(\{W_{\gamma}=0\}_{\gamma}/G\C^{\ast};\C)
\end{equation}

For an element $g\in\{g^{(1)},\dots,g^{(M)}\}$, write $g:= (g_1,\dots,g_N)$. Chiodo and Ruan showed that the contribution to the cohomology coming from the coset $g\C^{\ast}$ is
\begin{equation}\label{decomposition}
\bigoplus_{\lambda\in\bigcup_{j=1}^{N}\{\alpha \mid \alpha^{-w_j}=g_{j}\} } H^{\ast}(\{ W_{g\bar{\lambda}} =0\}_{g\bar{\lambda}}/G\C^{\ast};\C).
\end{equation}

Moreover, each summand in Equation (\ref{decomposition}) can be written explicitly as
\begin{equation}\label{hypercoho}
H^{\ast}(\{W_{g\bar{\lambda}}=0\}_{g\bar{\lambda}}/G\C^{\ast};\C)=H^{N_{g\bar{\lambda}}}(\C_{g\bar{\lambda}}^{N},W_{g\bar{\lambda}}^{+\infty};\C)^{G}\oplus\bigoplus_{r=0}^{N_{g\bar{\lambda}}-2}\C \cdot P^{r} \textbf{1}_{g\bar{\lambda}},
\end{equation}
where $P \textbf{1}_{g\bar{\lambda}}$ represents the intersection of the hyperplane class in $\mathbb{P}(w_1,\dots,w_N)$ with the hypersurface $\{W_{g\bar{\lambda}}=0\}_{g\bar{\lambda}}\subset\mathbb{P}(\boldsymbol{w}_\lambda)$.

Combining Equations (\ref{coset-decomposition}), (\ref{decomposition}) and (\ref{hypercoho}) we obtain a complete decomposition of the Chen-Ruan cohomology of the quotient stack $[\mathcal{X}_{W}/\widetilde{G}]$.

\begin{remark}
The first summand in Equation (\ref{hypercoho}) corresponds to the $\widetilde{G}$-invariant part of the primitive cohomology of the hypersurface $\{W_{g\bar{\lambda}}=0\}_{g\bar{\lambda}}\subset \mathbb{P}(\boldsymbol{w}_{g\bar{\lambda}})$.
\end{remark}
\begin{remark}\label{self-intersection}
The Chen-Ruan bi-degree of $P^{r} \textbf{1}_{g\bar{\lambda}}$ is given by $(r+a(g\bar{\lambda}),r+a(g\bar{\lambda}))$, where $a(g\bar{\lambda})$ is the age defined in Equation (\ref{ageshift}). It follows that in the modified Chen-Ruan state space of Equation (\ref{correspondence-equation}), this class will live in degree zero.
\end{remark}
\begin{remark}[Notation]
When $G=\langle J_{W}\rangle$, we only have contributions from the $\C^{\ast}$-coset corresponding to $g=1$. In this case, the sectors of the inertia orbifold of $\mathcal{X}_W$ are classified by elements of the form $\lambda=\exp (2\pi i f)$, with $f\in F:=\left \{\frac{k}{w_j} \middle | \text{ }0\leq k \leq {w_j-1}, \text{ }j=1,\dots,N\right\}$. To simplify notation, we denote $\textbf{1}_{\bar{\lambda}}$ simply by $\textbf{1}_f$.
\end{remark}

\subsubsection{Explicit decomposition of the FJRW state space:}
We now find an explicit decompostion of the FJRW state space defined in (\ref{FJRWspace}):
\[
H_{\text{FJRW}}(W,G)=\bigoplus_{\gamma\in G}H^{N_{\gamma}}(\C_{\gamma}^N, W_{\gamma}^{+\infty};\C)^{G}
\]

We first decompose $G$ into $M=|G|/d$ cosets of the form $g^{(1)}\langle J_{W}\rangle,\dots,g^{(M)}\langle J_{W}\rangle$, where the elements $\{g^{(1)},\dots, g^{(M)}\}\subset G$ are those of Equation (\ref{cosets}). Then, the FJRW state space can be written as
\[
H_{\text{FJRW}}(W,G)=\bigoplus_{i=1}^{M}\bigoplus_{k=0}^{d-1}H^{N_{g^{(i)}J_{W}^{k}}}(\C_{{g^{(i)}J_{W}^{k}}}^N, W_{{g^{(i)}J_{W}^{k}}}^{+\infty};\C)^{G}.
\]

For an element $g\in\{g^{(1)},\dots,g^{(M)}\}$, write $g= (g_1,\dots,g_N)$, as in the previous section. Also, write $J_{W}=(\xi_d^{w_{1}},\dots, \xi_d^{w_{N}})$ where $\xi_{d}:=\text{exp}(2\pi i/d)$. It is easy to see that $N_{gJ_{W}^{k}}=0$ (i.e. the sector $\mathcal{H}_{gJ_{W}^{k}}$ is narrow) if and only if $\xi_{d}^{k}$ does not belong to $\bigcup_{j=1}^N\{\lambda \mid \lambda^{-w_j}=g_j\}$. Thus, the total contribution coming from the $g \langle J_{W}\rangle$ coset is equal to
\begin{equation}\label{FJRW-decomposition}
\bigoplus_{\lambda\in \boldsymbol{\mu}_{d}\cap \bigcup_{j=1}^N\{\alpha \mid \alpha^{-w_j}=g_j\}} H^{N_{g\bar{\lambda}}}(\C_{g\bar{\lambda}}^{N}, W_{g\bar{\lambda}}^{+\infty};\C)\oplus\bigoplus_{\lambda\in \boldsymbol{\mu}_{d}\backslash \bigcup_{j=1}^N\{\alpha \mid \alpha^{-w_j}=g_j\}} \phi_{g\bar{\lambda}}\C
\end{equation}

where $\phi_{g\bar{\lambda}}:=1\in\mathcal{H}_{g\bar{\lambda}}\cong\C$, and we have identified the elements of $\langle J_{W}\rangle$ with $\bar{\lambda}$, for $\lambda\in \boldsymbol{\mu}_{d}$.
\begin{remark}
The first summand of Equation (\ref{FJRW-decomposition}) corresponds to broad sectors whereas the second summand is the contribution arising from narrow sectors.
\end{remark}
\begin{remark}\label{narrow}
The FJRW bi-degree of $\phi_{g\bar{\lambda}}$ is given by $\left(a(g\bar{\lambda})-\sum_{j=1}^N q_j ,a(g\bar{\lambda})-\sum_{j=1}^N q_j\right)$, where $a(g\bar{\lambda})$ is the age defined in Equation (\ref{age}). It follows that in the modified FJRW state space of Equation (\ref{correspondence-equation}), this class will live in degree zero.
\end{remark}
\begin{remark}[Notation]
When $G=\langle J_W\rangle$, the sectors of the FJRW state space are classified by powers of $J_W$. When this is the case, we denote $\phi_{J_W^{k+1}}$ simply by $\phi_k$.
\end{remark}
\subsubsection{A combinatorial diagram}

We now present a combinatorial diagram that will allow us to keep track of the different components of the CR-cohomology and the FJRW state space. This model was first introduced by Boissi\`{e}re, Mann, and Perroni (see \cite{BMP}) in order to compute the orbifold cohomology of weighted projective spaces. The model was later expanded by Chiodo and Ruan (see \cite[Section 5]{Chiodo-Ruan}.)

For each one of the cosets represented by the elements $g^{(1)},\dots, g^{(M)}$ defined in Equation (\ref{cosets}), we can define a diagram that will allows to keep track of the coset's contribution to both the CR cohomology and the FJRW state space. We now outline how to construct this diagram for an element $g^{(i)}$. For the sake of legibility, denote $g^{(i)}=(g_{1}^{(i)},\dots,g_{N}^{(i)})$ simply by $g=(g_1,\dots,g_N)$.

\begin{enumerate}
\item Draw a semi-infinite ray 
\begin{equation}\label{definition1}
\{\rho\nu \in\C \mid \rho\geq 0, \rho\in \R\}
\end{equation}
for every $\nu\in \boldsymbol{\mu}_d\cup \bigcup_{j=1}^{N}\{\lambda \mid \lambda^{-w_j}=g_{j}\}$ .
\\
\item Mark dots in position $j\nu$ if there exists  $1\leq j \leq N$ such that
\begin{equation}\label{definition2}
\nu^{-w_{j}}=g_j.
\end{equation}
\item Mark further dots in position $(N+1)\nu$ if
\begin{equation}\label{definition3}
\nu\in  \bigcup_{j=1}^{N}\{\lambda \mid \lambda^{-w_j}=g_{j}\} \backslash \boldsymbol{\mu}_d.
\end{equation}
\end{enumerate}

\begin{remark}
Note that by construction all dots lie on rays. Also,\\ $\nu\in  \bigcup_{j=1}^{N}\{\lambda \mid \lambda^{-w_j}=g_{j}\} \backslash \boldsymbol{\mu}_d$ if and only if $g\bar{\nu}=(g_1\nu^{w_1},\dots, g_N\nu^{w_N})\notin G$.
\end{remark}
\subsubsection{Interpretation of the diagram:}
\begin{enumerate}
\item A ray with angular coordinate $2\pi k/d$ represents the sector $\mathcal{H}_{gJ^k}$ of the FJRW state space. Note that a ray will carry no dots if and only if the corresponding FJRW sector is narrow (i.e. the fixed locus of $gJ^{k}$ is trivial.)\\
\item Consider the dot $j\nu$ lying on the ray $\{\rho\nu \in\C \mid \rho\geq 0, \rho\in \R\}$. We say the dot is an $extremal$ $dot$ if there is no other dot with higher radial coordinate on the same ray and it is an $internal$ $dot$ otherwise. 
\\
\item An extremal dot of the form $j\nu$ represents the primitive cohomology of the $g\bar{\nu}$-sector in the Chen-Ruan cohomology: $H^{prim}(\{W_{g\bar{\nu}}=0\}_{\mathbb{P}(\boldsymbol{w}_{g\bar{\nu}})}/\widetilde{G};\C)\cong H^{N_{g\bar{\nu}}}(\C_{g\bar{\nu}}^{N},W_{g\bar{\nu}}^{+\infty};\C)^{G}$.
\\
\item Internal dots on a ray of the form $\{\rho\nu \in\C \mid \rho\geq 0, \rho\in \R\}$ represent products of hyperplane classes in $H^{\ast}(\{W_{g\bar{\nu}}=0\}_{\mathbb{P}(\boldsymbol{w}_{g\bar{\nu}})}/\widetilde{G};\C)$, i.e. $\boldsymbol{1}_{g\bar{\nu}}$, $P\boldsymbol{1}_{g\bar{\nu}}$, $P^2\boldsymbol{1}_{g\bar{\nu}}$,$\dots$ 
\end{enumerate}

\begin{remark}
Let $R$ be the total number of rays in the diagram and $D$ the total number of dots in the diagram. From Equation (\ref{definition1}) it is easy to see that
\[
R=d+\left |  \bigcup_{j=1}^{N}\{\lambda \mid \lambda^{-w_j}=g_{j}\} \backslash \boldsymbol{\mu}_d \right |
\]
and from Equations (\ref{definition2}) and (\ref{definition3}), it follows that
\[
D=\sum_{j=1}^{N}w_j + \left |  \bigcup_{j=1}^{N}\{\lambda \mid \lambda^{-w_j}=g_{j}\} \backslash \boldsymbol{\mu}_d \right |.
\]

Hence, the difference between the total number of dots and rays in the diagram equals
\begin{equation}\label{difference}
D-R=\sum_{j=1}^{N}w_j - d=-\K.
\end{equation}
\end{remark}
\subsubsection{Proof of Theorem \ref{cohomological-correspondence}}
We now use the diagram described above to construct a correspondence between the contribution to the CR cohomology coming from the coset $g^{(i)}\C^{\ast}$ and the contribution to the FJRW state space coming from the coset $g^{(i)}\langle J_{W}\rangle$, for $1\leq i \leq M$. As in the previous section, we let $g^{(i)}=g$ to simplify notation.

$Step$ 1: There is a one-to-one correspondence between non-empty rays and extremal dots in the diagram. We must distinguish between two cases. In the first case, consider a ray of the form $\{\rho\nu \in\C \mid \rho\geq 0, \rho\in \R\}$ with $\nu\in \boldsymbol{\mu}_d$. From Equation (\ref{definition2}), it follows that the number of dots on this ray is $N_{g\bar{\nu}}$. The first $N_{g\bar{\nu}}-1$ dots correspond to intersections of the hyperplane class in $\{W_{g\bar{\nu}}=0\}_{\mathbb{P}(\boldsymbol{w}_{g\bar{\nu}})}$. The extremal dot corresponds to the primitive cohomology of $H^{\ast}(\{W_{g\bar{\nu}}=0\}_{\mathbb{P}(\boldsymbol{w}_{g\bar{\nu}})}/\widetilde{G};\C)$ which is isomorphic to the broad sector $\mathcal{H}_{g\bar{\nu}}$ of the FJRW state space by Equation (\ref{hypercoho}). For the second case, consider a ray with $\nu\in  \bigcup_{j=1}^{N}\{\lambda \mid \lambda^{-w_j}=g_{j}\} \backslash \boldsymbol{\mu}_d$. It follows that $g\bar{\nu}\notin G$ and thus, by Lemma \ref{dichotomy}, $\{W_{g\bar{\nu}}=0\}_{\mathbb{P}(\boldsymbol{w}_{g\bar{\nu}})}={\mathbb{P}(\boldsymbol{w}_{g\bar{\nu}})}$. From Equations (\ref{definition2}) and (\ref{definition3}) the number of dots on this ray is $N_{g\bar{\nu}}+1$. The first $N_{g\bar{\nu}}$ dots correspond to intersections of the hyperplane class in ${\mathbb{P}(\boldsymbol{w}_{g\bar{\nu}})}$. The extremal dot corresponds to the primitive cohomology of ${\mathbb{P}(\boldsymbol{w}_{g\bar{\nu}})}$, which is trivial. 

In this way, the contributions to Chen-Ruan cohomology arising from primitive cohomology precisely correspond to the contributions to FJRW theory coming from broad sectors. Note also that if we take a (p,q)-class in either state space, the difference $p-q$ cancels the age shifts of Equations (\ref{CRstate}) and (\ref{FJRWstate}), yielding a degree-preserving correspondence between the modified Chen-Ruan and FJRW state spaces. 

$Step$ 2: Since non-empty rays and extremal dots are exactly matched up, it follows that the difference between the total number of internal dots (which correspond to intersections of hyperplane classes) and the total number of empty rays (which correspond to narrow sectors) in the diagram is equal to $D-R=-\K$ by Equation (\ref{difference}). By Remarks \ref{self-intersection} and \ref{narrow}, the contributions coming from intersections of hyperplane classes in Chen-Ruan cohomology and from narrow sectors in FJRW theory live in degree zero in their respective modified state spaces. Hence, to get a one-to-one correspondence in degree zero, we need to add an extra component of dimension $|G|/d(\sum_{j}w_j - d)$ to $\bigoplus_{p}H_{\text{FJRW}}^{p,p}(W,G)$ (we add one component of dimension $\sum_j w_j -d$ for each one the $M=|G|/d$ cosets) in the Fano case, and an extra component of dimension $|G|/d(d-\sum_{j}w_j )$ to $ \bigoplus_{p}H_{\text{CR}}^{p,p}([\mathcal{X}_{W}/\widetilde{G}];\C)$ in the general type case. This accounts for the definition of $\mathcal{A}_{FJRW}^0(W,G)$ in Equation (\ref{modified-FJRW-state-space-Fano}), and for the definition of $\mathcal{A}_{\text{CR}}^{0}([\mathcal{X}_{W}/\widetilde{G}];\C)$ in Equation (\ref{modified-Chen-Ruan-state-space-gt}).
\begin{flushright}$\square$\end{flushright}


\subsection{Examples}\label{examples}
\subsubsection{A smooth cubic four-fold} Let $\mathcal{X}^3$ be the smooth cubic four-fold in $\mathbb{P}^5$ defined by the vanishing locus $\{W=x_1^3+x_2^3+x_3^3+x_4^3+x_5^3+x_6^3=0\}$. We consider the case $G=\langle J_{W}\rangle$. The Fano index of this cubic four-fold is  $-\kappa_{\mathcal{X}^3}=\sum_i w_i -d =3$. The Hodge diamond of $\mathcal{X}^3$ has the form
\[
\begin{tikzpicture}

\node[] at (0,2) {1};

\node[] at (-0.5,1.5) {0};
\node[] at (0.5,1.5) {0};

\node[] at (-1,1) {0};
\node[] at (0,1) {1};
\node[] at (1,1) {0};

\node[] at (-1.5,0.5) {0};
\node[] at (-0.5,0.5) {0};
\node[] at (0.5,0.5) {0};
\node[] at (1.5,0.5) {0};

\node[] at (0,0) {21};
\node[] at (-1,0) {1};
\node[] at (1,0) {1};
\node[] at (-2,0) {0};
\node[] at (2,0) {0};

\end{tikzpicture}
\]

(see, for example, \cite{Hassett}). From this, we can directly compute the dimension of the modified Chen-Ruan state space: 

\[
\text{dim}\mathcal{A}_{\text{CR}}^n(\mathcal{X}^3;\C)=
\begin{cases}
h^{0,0}+h^{1,1}+h^{2,2}+h^{3,3}+h^{4,4}=25 & \text{if } n=0\\
h^{3,1}=1 & \text{if } n=2\\
h^{1,3}=1 & \text{if } n=-2\\
0 & \text{otherwise.}
\end{cases}
\]

The relevant information for the FJRW state space in contained in the following table:
\begin{center}
\begin{tabular}{ c || c | c | c | c | c | c || c | c}
                    
  $J_{W}^k$ & $x_1$ & $x_2$ & $x_3$ & $x_4$ & $x_5$ & $x_6$ & $\text{deg}_{\text{FJRW}}$ & ($h^{p,q}$| $p+q=\text{deg}_{\text{FJRW}}$) \\
  \hline     
   & & & & & & & & \\
  $J_{W}^0$ & 0 & 0 & 0 & 0 & 0 & 0 & 2 & $h^{2,0}=1,h^{1,1}=20,h^{0,2}=1$\\
  $J_{W}^1$ & 1 & 1 & 1 & 1 & 1 & 1 & 0 & $h^{0,0}=1$\\
  $J_{W}^2$ & 2 & 2 & 2 & 2 & 2 & 2 & 4 & $h^{2,2}=1$\\
  \end{tabular}
\end{center}

and thus, the dimension of the modified FJRW state space is given by
\[
\text{dim}\mathcal{A}_{\text{FJRW}}^n(W,\langle J_{W}\rangle)=
\begin{cases}
h^{0,0}+h^{1,1}+h^{2,2}+3=25 & \text{if } n=0\\
h^{2,0}=1 & \text{if } n=2\\
h^{0,2}=1& \text{if } n=-2\\
0 & \text{otherwise}
\end{cases}
\]
which gives the desired degree preserving isomorphism of modified state spaces.
\begin{figure}[h]

\resizebox{6.2cm}{5.5cm}{

\begin{tikzpicture}
\draw[dashed] (5,5) circle (6);
\draw[dashed] (5,5) circle (5);
\draw[dashed] (5,5) circle (4);
\draw[dashed] (5,5) circle (3);
\draw[dashed] (5,5) circle (2);
\draw[dashed] (5,5) circle (1);

\draw (5,5) -- (12,5);
\draw (5,5) -- (1.5,10.8);
\draw (5,5) -- (1.5,-0.8);

\fill[black] (6,5) circle (0.2);
\fill[black] (7,5) circle (0.2);
\fill[black] (8,5) circle (0.2);
\fill[black] (9,5) circle (0.2);
\fill[black] (10,5) circle (0.2);
\fill[black] (11,5) circle (0.2);

\end{tikzpicture}
}
\caption{Combinatorial diagram for the smooth cubic four-fold.}\label{cubic}
\end{figure}
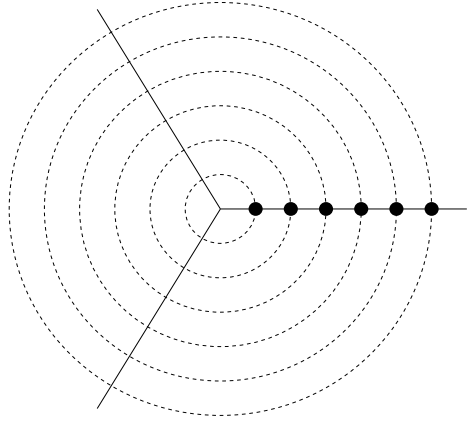

 Figure \ref{cubic} shows the combinatorial diagram for the cubic four-fold. The empty rays correspond to the narrow FJRW sectors $J_W^1$ and $J_W^2$. Note that the difference between the number of internal dots and the number of empty rays is precisely $-\kappa_{\mathcal{X}^3} =3$.
\subsubsection{A non-Gorenstein example}\label{non-gorenstein}
Consider the degree $8$ surface $\mathcal{X}^8$ cut out by $\{W=x_1^4+x_1x_2^2+x_3^2+x_4^2=0\}$ inside $\mathbb{P}(2,3,4,4)$. As in the previous example, let $G=\langle J_W\rangle$. Note that not all the weights divide the degree of the surface and therefore, this is a non-Gorenstein stack in weighted projective space. The Fano index of this surface is given by  $-\kappa_{\mathcal{X}^8}=13-8=5$.  We now compute the cohomology of the different sectors of the inertia stack of $\mathcal{X}^8$. Let $\zeta_4=\text{exp}(\frac{2\pi i}{4})$ and $\zeta_3=\text{exp}(\frac{2\pi i}{3})$. 
\begin{itemize}
\item Untwisted sector: its cohomology is the direct sum of three components of bi-degrees $(0,0)$, $(1,1)$ and $(2,2)$, arising from intersections with the hyperplane class, together with the primitive part of the cohomology.  This last component may be computed as the $\langle J_{W}\rangle$-invariant part of the Milnor ring of the singularity $W:\C^4\rightarrow \C$ (see, for example, \cite{Dolgachev, Steenbrink}). It has only one component, of bi-degree $(1,1)$.\\

\item $\zeta_4 $-sector: this sector corresponds to the locus $\{x_3^2+x_4^2=0\}\subset \mathbb{P}(4,4)$. This gives two orbifold points, both of CR bi-degree $\left(\frac{5}{4},\frac{5}{4}\right)$.\\

\item $\zeta_3,\zeta_3^2 $-sector: their cohomology is given by $H^{\ast}([\text{pt}/\mathbb{Z}_3];\C)\cong H^{\ast}(\text{pt};\C)=\C$ in bi-degrees $\left(\frac{2}{3},\frac{2}{3}\right)$ and $\left(\frac{4}{3},\frac{4}{3}\right)$, respectively.\\

\item $-1$-sector: this sector corresponds to the vanishing locus $\{x_1^4+x_3^2+x_4^2=0\}\subset\mathbb{P}(2,4,4)$. Its cohomology is isomorphic to $H^{\ast}(\mathbb{P}^1;\C)=\C\oplus\C$, in CR bi-degrees $\left(\frac{1}{2},\frac{1}{2}\right)$ and $\left(\frac{3}{2},\frac{3}{2}\right)$.\\

\item $\zeta_4^3 $-sector: as in the $\zeta_4$-sector, the cohomology corresponds to the contribution of two orbifold points of bi-degree $\left(\frac{3}{4},\frac{3}{4}\right)$.
\end{itemize}

Note that all contributions to the CR cohomology have bi-degrees of the from $(p,p)$. It follows that the dimension of modified CR state space is given by
\[
\text{dim}\mathcal{A}_{\text{CR}}^{n}(\mathcal{X}^8;\C)=
\begin{cases}
12 & \text{if } n=0 \\
0 & \text{otherwise.}
\end{cases}
\]

The following table contains all the relevant information on the LG side:
\begin{center}
\begin{tabular}{ c || c | c | c | c || c | c}
                    
  $J_{W}^k$ & $x_1$ & $x_2$ & $x_3$ & $x_4$ & $\text{deg}_{\text{FJRW}}$ & ($h^{p,q}$| $p+q=\text{deg}_{\text{FJRW}}$) \\
  \hline     
   & & & & & & \\
  $J_{W}^0$ & 0 & 0 & 0 & 0 & 3/4 & $h^{\frac{3}{8},\frac{3}{8}}=1$\\
  $J_{W}^1$ & 2 & 3 & 4 & 4 & 0 & $h^{0,0}=1$\\
  $J_{W}^2$ & 4 & 6 & 0 & 0 & 5/4 & $h^{\frac{5}{8},\frac{5}{8}}=1$\\
  $J_{W}^3$ & 6 & 1 & 4 & 4 & $1/2$ & $h^{\frac{1}{4},\frac{1}{4}}=1$\\
  $J_{W}^4$ & 0 & 4 & 0 & 0 &  & \\
  $J_{W}^5$ & 2 & 7 & 4 & 4 & 1 & $h^{\frac{1}{2},\frac{1}{2}}=1$\\
  $J_{W}^6$ & 4 & 2 & 0 & 0 & 1/4 & $h^{\frac{1}{8},\frac{1}{8}}=1$\\
  $J_{W}^7$ & 6 & 5 & 4 & 4 & 3/2 & $h^{\frac{3}{4},\frac{3}{4}}=1$\\
\end{tabular} 
\end{center} 
  From this we can compute the dimension of the modified FJRW state space:
  \[
  \text{dim}\mathcal{A}_{\text{FJRW}}^0(W,\langle J_{W}\rangle)=h^{\frac{3}{8},\frac{3}{8}}+h^{0,0}+h^{\frac{5}{8},\frac{5}{8}}+h^{\frac{1}{4},\frac{1}{4}}+h^{\frac{1}{2},\frac{1}{2}}+h^{\frac{1}{8},\frac{1}{8}}+h^{\frac{3}{4},\frac{3}{4}}+5=12,
  \]
  providing the desired isomorphism of modified state spaces.
  
\begin{figure}[h]
\resizebox{7cm}{6cm}{
\begin{tikzpicture}
\draw[dashed] (5,5) circle (5);
\draw[dashed] (5,5) circle (4);
\draw[dashed] (5,5) circle (3);
\draw[dashed] (5,5) circle (2);
\draw[dashed] (5,5) circle (1);
\draw (-1,5) -- (11,5);
\draw (5,-0.5) -- (5,10.5);
\draw (1,1) -- (9,9);
\draw (1,9) -- (9,1);

\draw[red] (5,5) -- (2,10.1);
\draw[red] (5,5) -- (2,-0.1);

\fill[black] (6,5) circle (0.2);
\fill[black] (7,5) circle (0.2);
\fill[black] (8,5) circle (0.2);
\fill[black] (9,5) circle (0.2);

\fill[black] (5,8) circle (0.2);
\fill[black] (5,9) circle (0.2);

\fill[black] (4,5) circle (0.2);
\fill[black] (2,5) circle (0.2);
\fill[black] (1,5) circle (0.2);

\fill[black] (5,2) circle (0.2);
\fill[black] (5,1) circle (0.2);

\fill[black] (2.5,9.3) circle (0.2);
\fill[black] (4.5,5.9) circle (0.2);

\fill[black] (2.5,0.7) circle (0.2);
\fill[black] (4.5,4.1) circle (0.2);

\end{tikzpicture}
}
\caption{Combinatorial diagram for a degree $8$ surface in $\mathbb{P}(2,3,4,4)$.}\label{degree-8}
\end{figure}
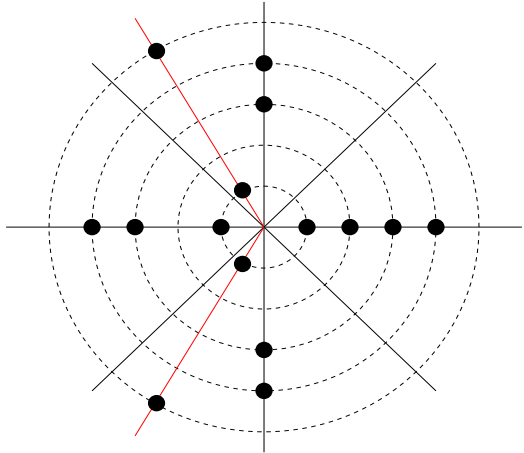

Figure \ref{degree-8} shows the combinatorial diagram for $\mathcal{X}^8$. The difference between the number of internal dots and the number of empty rays is $-\kappa_{\mathcal{X}^8}=5$. The rays in red represent the "non-Gorenstein" contributions to the Chen-Ruan cohomology.

\subsubsection{An example with $G\neq \langle J_W\rangle$} Consider the degree $6$ orbicurve $\mathcal{X}^6$ defined by the vanishing locus of $W=x_1^6+x_2^2+x_3^2$ inside $\mathbb{P}(1,3,3)$. The group of diagonal symmetries $\text{Aut}(W)$ contains the element $g:=(-1,-1,1)$. Define $G<\text{Aut}(W)$ by
\[
G:=\{g^lJ_{W}^k \mid 0\leq l \leq 1,\text{ }0\leq k \leq 5\}.
\]
G is a non-cyclic group of order $12$ and $\widetilde{G}=G/\langle J_W\rangle \cong \mathbb{Z}_2$. We now compute the Chen-Ruan cohomology of the quotient stack $[\mathcal{X}^6/\mathbb{Z}_2]$. We can compute the cohomology of inertia stack of $[\mathcal{X}^6/\mathbb{Z}_2]$ by using the decomposition found in Equations (\ref{decomposition}) and (\ref{hypercoho}). The non-empty sectors are:
\begin{itemize}
\item Untwisted sector: we need to compute $H^{\ast}(\mathcal{X}^6;\C)^{\widetilde{G}}$. As explained in \cite{Dolgachev}, $H^{\ast}(\mathcal{X}^6;\C)\cong H^{\ast}(\mathbb{P}^1;\C)$, and thus, the only contributions come from hyperplane classes of CR degrees $(0,0)$ and $(1,1)$.\\
\item $J_W^2$-sector: the contribution to this sector comes from $H^{\ast}([\{x_2^2+x_3^2=0\}_{\mathbb{P}(3,3)}/\mathbb{Z}_2];\C)$. The solution to $x_2^2+x_3^2=0$ in $\mathbb{P}(3,3)$ can be represented by the two orbifold points $\{(-i,1),(i,1)\}$. However, the action of $\mathbb{Z}_2$ identifies these two points and we end up with a one dimensional contribution to the cohomology, of CR degree $\left(\frac{1}{3},\frac{1}{3}\right)$.\\
\item $J_W^4$-sector: this case is similar to the $J_W^2$-sector. We have a one dimensional contribution of CR degree $\left(\frac{2}{3},\frac{2}{3}\right)$.\\
\item $g J_W^3$-sector: this sector corresponds to $[ \{x_1^6+x_2^2=0\}_{\mathbb{P}(1,3)}/\mathbb{Z}_2]$. It can be represented by the two orbifold points $(1,i)$ and $(1,-i)$, which are not identified by the action of $\mathbb{Z}_2$. Hence, we end up with two contributions to the cohomology, both of CR degree $\left(\frac{1}{2},\frac{1}{2}\right)$.
\end{itemize}
From this analysis, we find that the dimension of the modified Chen-Ruan state space is given by
\[
\text{dim}\mathcal{A}_{\text{CR}}^{n}([\mathcal{X}^6/\mathbb{Z}_2];\C)=
\begin{cases}
6 & \text{if } n=0 \\
0 & \text{otherwise.}
\end{cases}
\]

The corresponding information on the LG side is provided in the following tables:
\begin{center}
\begin{tabular}{ r || c | c | c || c | c }
  $g^lJ_{W}^k$ & $x_1$ & $x_2$ & $x_3$ & $\text{deg}_{\text{FJRW}}$ & ($h^{p,q}$| $p+q=\text{deg}_{\text{FJRW}}$) \\
  \hline     
   & & & & &  \\
  $J_{W}^0$ & 0 & 0 & 0 &  & \\
  $J_{W}^1$ & 1 & 3 & 3 & 0 & $h^{0,0}=1$\\
  $J_{W}^2$ & 2 & 0 & 0 &  & \\
  $J_{W}^3$ & 3 & 3 & 3 & 2/3 & $h^{\frac{1}{3},\frac{1}{3}}=1$\\
  $J_{W}^4$ & 4 & 0 & 0 &  & \\
  $J_{W}^5$ & 5 & 3 & 3 & 4/3 & $h^{\frac{2}{3},\frac{2}{3}}=1$\\
\end{tabular}
\end{center}

\begin{center}
\begin{tabular}{ r || c | c | c || c | c }
  $g^lJ_{W}^k$ & $x_1$ & $x_2$ & $x_3$ & $\text{deg}_{\text{FJRW}}$ & ($h^{p,q}$| $p+q=\text{deg}_{\text{FJRW}}$) \\
    \hline     
  & & & & &  \\
  $gJ_{W}^0$ & 3 & 3 & 0 &  & \\
  $gJ_{W}^1$ & 4 & 0 & 3 &  & \\
  $gJ_{W}^2$ & 5 & 3 & 0 &  & \\
  $gJ_{W}^3$ & 0 & 0 & 3 & 2/3 & $h^{\frac{1}{3},\frac{1}{3}}=1$\\
  $gJ_{W}^4$ & 1 & 3 & 0 & & \\
  $gJ_{W}^5$ & 2 & 0 & 3 &  & \\
\end{tabular} 
\end{center}
Thus, dimension of the modified FJRW state space is given by
  \[
\text{dim}\mathcal{A}_{\text{FJRW}}^n(W,G)=
\begin{cases}
4 +2\times1=6 & \text{if } n=0\\
0 & \text{otherwise}
\end{cases}
\]
from which we obtain a degree preserving isomorphism.
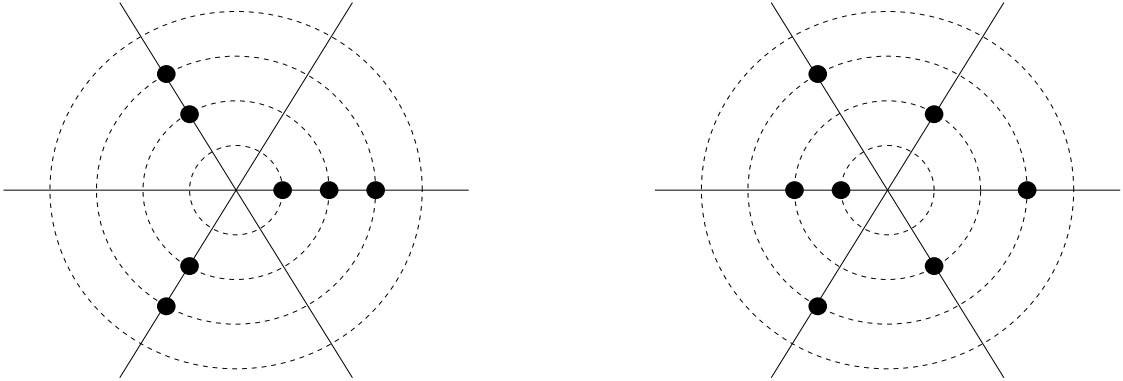
\begin{figure}[h]
\resizebox{15cm}{5cm}{
\begin{tikzpicture}

\draw[dashed] (-10,5) circle (4);
\draw[dashed] (-10,5) circle (3);
\draw[dashed] (-10,5) circle (2);
\draw[dashed] (-10,5) circle (1);
\draw (-15,5) -- (-5,5);

\draw (-7.5,0.8) -- (-12.5,9.2);
\draw (-7.5,9.2) -- (-12.5,0.8);

\fill[black] (-9,5) circle (0.2);
\fill[black] (-8,5) circle (0.2);
\fill[black] (-7,5) circle (0.2);

\fill[black] (-11,6.7) circle (0.2);
\fill[black] (-11.5,7.6) circle (0.2);

\fill[black] (-11.5,2.4) circle (0.2);
\fill[black] (-11,3.3) circle (0.2);


\draw[dashed] (4,5) circle (4);
\draw[dashed] (4,5) circle (3);
\draw[dashed] (4,5) circle (2);
\draw[dashed] (4,5) circle (1);

\draw (-1,5) -- (9,5);
\draw (6.5,0.8) -- (1.5,9.2);
\draw (6.5,9.2) -- (1.5,0.8);

\fill[black] (7,5) circle (0.2);
\fill[black] (3,5) circle (0.2);
\fill[black] (2,5) circle (0.2);

\fill[black] (2.5,7.6) circle (0.2);
\fill[black] (5,6.7) circle (0.2);

\fill[black] (2.5,2.4) circle (0.2);
\fill[black] (5,3.3) circle (0.2);

\end{tikzpicture}
}
\caption{Combinatorial diagram for the quotient stack $[\mathcal{X}^6/\mathbb{Z}_2]$.}\label{degree-6}
\end{figure}

Figure \ref{degree-6} illustrates the combinatorial diagrams for the two $\langle J_W\rangle$-cosets of $G$. In each of the diagrams the difference between the number of internal dots and the number of internal rays is 1.



\section{Quantum Correspondence: computing FJRW $I$-function}\label{quantum-theory}

\subsection{Genus Zero FJRW Theory}

In this section we define the genus zero FJRW theory. We also compute the FJRW $I$-function by using the formalism of Givental. All the results found in this section assume that $G=\langle J_{W}\rangle$ and that the Gorenstein condition $w_j | d$ for $j=1,\dots, N$ is satisfied. The presentation in this section follows the notation found in \cite{CIR}. For a more general construction of FJRW theory, the reader is directed to \cite{FJRa}.
\subsubsection{Pointed orbicurves and orbifold line bundles:}
\begin{definition}
A \textit{n-pointed orbicurve} $(\mathcal{C},p_1,\dots,p_n)$ is a proper and connected Deligne-Mumford stack of dimension one with at worst nodal singularities and which has $n$ marks $p_1,\dots,p_n$ on the smooth locus such that
\begin{enumerate}
\item the curve has possibly non-trivial stabilizers only at the marks and nodes;
\item the nodes are \textit{balanced}, i.e. in the local picture of the node $\{xy=0\}$, the action of the isostropy group $\boldsymbol{\mu}_r$ is given by
\[
(x,y)\mapsto (\zeta_r x,\zeta_r^{-1} y)
\]
\end{enumerate}
where $\zeta_r$ is a preferred generator of $\boldsymbol{\mu}_r$. The coarse underlying curve of the pointed orbicurve is denoted by $(|\mathcal{C}|,|p_1|,\dots, |p_n|)$ and its natural projection by $\rho : \mathcal{C}\rightarrow |\mathcal{C}|$. This projection is a flat morphism so, in particular, if $\mathcal{L}$ is an orbifold line bundle on $\mathcal{C}$, its pushforward $\rho_{\ast}\mathcal{L}$ is a line bundle over $|\mathcal{C}|$.
\end{definition}

\begin{definition}
For a positive integer $d$, an $n$-pointed orbicurve $(\mathcal{C},p_1,\dots,p_n)$ is said to be $d$-\textit{stable} if its coarse underlying curve $(|\mathcal{C}|,|p_1|,\dots, |p_n|)$ is stable and the stabilizers at the marks and nodes are isomorphic to $\boldsymbol{\mu}_d$.
\end{definition}

Given an orbifold line bundle $\mathcal{L}$ on a $d$-stable, $n$-pointed orbicurve $(\mathcal{C},p_1,\dots,p_n)$, we can define the \textit{age} of $\mathcal{L}$ at a node in the following way. Let  $\{xy=0\}\times\C$ be a local trivialization of $\mathcal{L}$ over a node $\sigma$ on $\mathcal{C}$. Then the preferred generator $\zeta_d$ of the isotropy group $\boldsymbol{\mu}_d$ acts on the local trivialization in the following way
\[
(x,y,v)\mapsto (\zeta_d x,\zeta_d^{-1}y,\zeta_d^k v), 
\]
where $k\in\{0,\dots,d-1\}$. The age of $\mathcal{L}$ at $\sigma$ is defined as the rational number
\begin{equation}
\mult := \frac{k}{d} \in [0,1).
\end{equation}
The age of $\mathcal{L}$ at a marked point in defined similarly.

\subsubsection{W-structures}

Given a $d$-stable, $n$-pointed orbicurve $(\mathcal{C},p_1,\dots,p_n)$, we define the invertible sheaf $\omega_{\mathcal{C},\text{log}}$ on $\mathcal{C}$ as the pullback of the dualizing sheaf $\omega_{\mathcal{|C|}}$ on $|\mathcal{C}|$ twisted at the points $|p_1|,\dots,|p_n|$, i.e.
\begin{equation}
\omega_{\mathcal{C},\text{log}}:=\rho^{\ast}\omega_{|\mathcal{C}|,\text{log}}:=\rho^{\ast}(\omega_{|\mathcal{C}|}(|p_1|+\dots +|p_n|)).
\end{equation}
\begin{definition}
A $d$-\textit{spin} structure on a pointed orbicurve $\mathcal{C}$ is an orbifold line bundle $\mathcal{L}\rightarrow\mathcal{C}$ together with an isomorphism $\phi:\mathcal{L}^{\otimes d}\IsoTo\omega_{\mathcal{C},\text{log}} $.

Given a quasi-homogeneous polynomial $\W$ of the form described in Equation (\ref{polynomial}), a $W$-\textit{structure} on a pointed orbicurve $\mathcal{C}$ is a collection of orbifold line bundles $\mathcal{L}_1,\dots, \mathcal{L}_N$ together with isomorphisms
\begin{equation}
\varphi_i:\bigotimes_{j=1}^N\mathcal{L}_j^{b_{i,j}}\IsoTo\omega_{\mathcal{C},\text{log}},\quad i=1,\dots,s.
\end{equation}
\end{definition}
A $d$-spin structure $\mathcal{L}\rightarrow\mathcal{C}$ gives rise to a $W$-structure in a natural way. Set $\mathcal{L}_j=\mathcal{L}^{\otimes w_j}$, by the quasi-homogeneity of $\W$ we get a natural isomorphism
\[
\bigotimes_{j=1}^N\mathcal{L}_j^{b_{i,j}}\cong\mathcal{L}^{\otimes d}\cong\omega_{\mathcal{C},\text{log}}.
\]
It is important to note that not all $W$-structures arise in this way but, in this paper, we restrict ourselves to the case in which $W$-structures arise from $d$-spin structures.
\begin{remark}
Given a $d$-spin structure $\mathcal{L}\rightarrow\mathcal{C}$ and a marked point $p\in \mathcal{C}$ with $\text{age}_{p}(\mathcal{L})=k/d$, the prefered generator $\zeta_{d}$ of $\boldsymbol{\mu}_d$ acts on the natural $W$-structure $(\mathcal{L}^{w_1},\dots ,\mathcal{L}^{w_N})$ as $(\zeta_d^{w_1k},\dots , \zeta_d^{w_Nk})=J_{W}^k$. For this reason we associate such a point $p\in\mathcal{C}$ with the FJRW-sector $\mathcal{H}_{J^k}$.
\end{remark}
\subsubsection{The moduli space}

Given integers $0\leq k_1,\dots, k_n\leq d-1$, we define$ \text{\calligra{W}}_{0,n}^{\quad d}(k_1,\dots, k_n)$ to be the moduli stack of $d$-stable orbicurves $\mathcal{C}$ of genus zero with $n$ marked points $p_1,\dots,p_n$ and endowed with a $d$-spin structure $\mathcal{L}\rightarrow \mathcal{C}$ such that $\text{age}_{p_i}(\mathcal{L})=  \genfrac{\langle}{\rangle}{}{}{k_i+1}{d}$ for $i=1,\dots,n$, i.e.
\begin{equation}
\begin{split}
\text{\calligra{W}}_{0,n}^{\quad d}(k_1,\dots,& k_n):=\\
&\left\{(\mathcal{C},p_1,\dots,p_n;\mathcal{L};\varphi: \mathcal{L}^{\otimes d}\IsoTo \w )\middle |  \text{  age}_{p_i}(\mathcal{L})=  \genfrac{\langle}{\rangle}{}{}{k_i+1}{d}\right\}/\text{isom.}
\end{split}
\end{equation}

Forgetting the $W$-structure and the orbifold structure gives a morphism
\[
\text{st}:\Wmod\rightarrow \overline{\mathcal{M}}_{0,n}.
\]
\begin{remark}
The stack $\Wmod$ corresponds to the stack $\mathcal{R}_d(e^{2\pi i \Theta_1},\dots, e^{2\pi \Theta_n})$ defined in \cite{Chiodo-Ruan2} as soon as $\Theta_i=\left\langle \frac{k_i+1}{d}\right \rangle$ for $i=1,\dots, n$. We have slightly changed the notation to obtain simpler formulas for the computations of FJRW invariants.
\end{remark}

\begin{lemma}\label{moduli-lemma}
Let $n>0$. The stack $\Wmod$ is smooth, proper, and of Deligne-Mumford type. It is non-empty if and only if $\mathcal{X}(\mathcal{L})=1+\frac{(n-2)}{d}-\sum_{i=1}^n  \genfrac{\langle}{\rangle}{}{}{k_i+1}{d}$ is an integer.
\end{lemma}
\begin{proof}
The first statement follows from \cite{Chiodo}. The second statement is a consequence of the Riemann-Roch theorem for orbicurves (see \cite{AGV}.)
\end{proof}
As in the case of the moduli space of stable curves, we have a universal orbicurve $\pi:\text{\calligra{C}}\longrightarrow \Wmod$. We also have a universal $d$-spin structure $\text{\calligra{L}}\longrightarrow\text{\calligra{C}}$.

\subsubsection{Enumerative geometry of FJRW theory}
We now define the invariants for FJRW theory. They are defined in terms of the virtual fundamental class constructed in \cite{FJRb}. Given integers $0\leq k_1,\dots , k_n \leq d-1$, the virtual cycle $[\Wmod]^{\text{vir}}$ lies in
\[
H_{2D(k_1,\dots,k_n)}(\Wmod ;\C)\otimes \bigotimes_{i=1}^n H_{N_{J^{k_i+1}}}(\C_{J^{k_i+1}}^N,W_{J^{k_i+1}}^{+\infty};\C)^{\langle J\rangle},
\]
where $D(k_1,\dots ,k_n) = n-3+\sum_{j=1}^N \mathcal{X}(\mathcal{L}^{\otimes w_j})$, with $\mathcal{L}$ a $d$-spin structure coming from $\Wmod$. By regarding the above relative homology as dual to the FJRW sectors $\mathcal{H}_{J^{k_i+1}}$, the virtual cycle defines a linear map
\begin{equation}
\begin{split}
\bigotimes_{i=1}^{n}&\mathcal{H}_{J^{k_i+1}}\longrightarrow H_{2D(k_1,\dots,k_n)}(\Wmod ;\C) \\
&\otimes_i \alpha_i \mapsto [\Wmod]^{\text{vir}}\cap \prod_{i=1}^n \alpha_i.
\end{split}
\end{equation}
For non-negative integers $a_1,\dots ,a_n$ and state space elements $\alpha_i\in \mathcal{H}_{J^{k_i+1}}$, $i=1,\dots, n$, the genus zero FJRW invariants are defined as
\begin{equation}\label{invariants}
\langle \tau_{a_1}(\alpha_1),\dots, \tau_{a_n}(\alpha_n)\rangle_{0,n}^{\text{FJRW}}:= \biggl ( [\Wmod]^{\text{vir}}\cap \prod_{i=1}^n \alpha_i \biggr )\cap \prod_{i=1}^n \widetilde{\psi}_{i}^{a_i}
\end{equation}
where the psi-classes $\widetilde{\psi}$ are defined via pullback under the morphism $\text{st}:\Wmod\rightarrow \overline{\mathcal{M}}_{0,n}$ from the usual psi-classes on $ \overline{\mathcal{M}}_{0,n}$. To simplify the notation, we write $\alpha_i$ instead of $\tau_{0}(\alpha_i).$
\begin{remark}
It is possible to generalize the above construction (i.e. moduli space, virtual cycle and invariants) to all genera. For details, we refer the reader to \cite{FJRa, FJRb}.
\end{remark}

\subsubsection{Restricting the invariants to narrow sectors}
The computation of FJRW invariants can be greatly simplified if we restrict them to narrow sectors. Define the set $\textbf{Nar}\subset \{0,\dots,d-1\}$ as
\begin{equation}\label{narrow-set}
\begin{split}
\textbf{Nar}:=&\left\{ k\in\{0,\dots,d-1\}\middle | \text{ }  \mathcal{H}_{J_W^{k+1}} \text{ is narrow} \right\}\\
& \left\{ k\in\{0,\dots,d-1\}\middle |\text{ }  (k+1)w_j\notin d\mathbb{Z}, \text{ for all } j=1,\dots,N \right\}.
\end{split}
\end{equation}
We denote the restriction of the FJRW state space to narrow sectors by 
\[
H_{\text{FJRW}}^{nar}(W):=\bigoplus_{k\in \textbf{Nar}}\mathcal{H}_{J_W^{k+1}}.
\]
Let $\pi:\text{\calligra{C}}\longrightarrow \Wmod$ be the universal curve and $\text{\calligra{L}}\longrightarrow\text{\calligra{C} }$  be the universal $d$-spin structure mentioned in the previous section. Define the following virtual bundle known as the \textit{obstruction bundle} 
\[
-\R\pi_{\ast}\bigoplus_{j=1}^N\text{\calligra{L } }^{\otimes w_j}=R^1\pi_{\ast}\bigoplus_{j=1}^N\text{\calligra{L } }^{\otimes w_j}-R^0\pi_{\ast}\bigoplus_{j=1}^N\text{\calligra{L } }^{\otimes w_j}.
\]
The following lemma (\cite[Lemma 2.3]{CIR}) shows that when restricting to classes coming from narrow sectors, the obstruction bundle becomes an honest vector bundle.

\begin{lemma}\label{obstruction}
Suppose $\mathcal{H}_{J^{k_i+1}}$ is a narrow sector for $i=1,\dots, n$. Then $H^0(\mathcal{C},\mathcal{L}^{\otimes w_j})$ vanishes for $j=1,\dots, N$ and for all $(\mathcal{C},p_1,\dots,p_n;\mathcal{L})$ in $\Wmod$. As a consequence, the obstruction bundle becomes the locally free bundle $R^1\pi_{\ast}\bigoplus_{j=1}^N\text{\calligra{L } }^{\otimes w_j}$. 
\end{lemma}

Using the concavity axiom described in \cite{FJRa} Equation (57), we can write the genus zero FJRW invariants in terms of the top Chern class of the dual of the obstruction bundle. Thus, if $\alpha_i\in \mathcal{H}_{J^{k_i+1}}$ are narrow classes for $i=1,\dots,n$, the genus zero FJRW invariants defined in Equation (\ref{invariants}) can be expressed as
\begin{equation}
\langle \tau_{a_1}(\alpha_1),\dots, \tau_{a_n}(\alpha_n)\rangle_{0,n}^{\text{FJRW}} =\int_{[\Wmod]}\prod_{i=1}^n \widetilde{\psi}_{i}^{a_i} \text{ }\cup \prod_{j=1}^N c_{\text{top}}\biggl (\biggl ( R^1\pi_{\ast}\text{\calligra{L } }^{\otimes w_j}\biggr )^{\ast}\biggr )
\end{equation}
where $[\Wmod] $ is the standard fundamental class of the moduli space. For dimensional reasons, the genus zero invariants will vanish unless $D(k_1,\dots, k_n)=\sum_i a_i$. For narrow classes this condition is equivalent to
\begin{equation}
\sum_{i=1}^n (a_i+\frac{1}{2}\text{deg}_{\text{FJRW}}(\alpha_i))=n-3+N-2\sum_{j=1}^N \frac{w_j}{d}.
\end{equation}

\subsubsection{The extended FJRW state space}

We define an \textit{extension} of the space of FJRW narrow sectors to be
\begin{equation}\label{definition-extended-space}
H_{\text{FJRW}}^{ext}(W):=\bigoplus_{k\in \textbf{Nar}}\mathcal{H}_{J_W^{k+1}}\oplus\bigoplus_{\substack{0\leq k \leq d-1\\ k\notin \textbf{Nar}}}\C \phi_k.
\end{equation}
If $k\in \textbf{Nar}$, define $\phi_{k}:=\textbf{1}\in\C\cong \mathcal{H}_{J_W^{k+1}}$. The extended FJRW state space can then be written as
\[
H_{\text{FJRW}}^{ext}(W)=\bigoplus_{k\in \textbf{Nar}}\C\phi_k \oplus\bigoplus_{\substack{0\leq k \leq d-1\\k\notin \textbf{Nar}}}\C \phi_k=\bigoplus_{k=0}^{d-1} \C\phi_k.
\]
The extended state space carries a natural grading
\begin{equation}
\text{deg}_{\text{FJRW}}(\phi_k)=2\sum_{j=1}^N  \left\langle \frac{kw_j}{d}\right\rangle ,
\end{equation}
and a natural pairing given by
\begin{equation}\label{grading}
(\phi_i,\phi_j):=
\begin{cases}
1 &\text{if } i+j+2\equiv 0 \text{ (mod d)}\\
0 &\text{otherwise}.
\end{cases}
\end{equation}
The additional states $\phi_k$ with $k\notin \textbf{Nar}$ will play the role of place-holders in the extended theory. Thus, in the extended theory, we want to define invariants that will vanish as soon as one of the entries is not narrow and that will equal the original FJRW invariants if all the insertions correspond to narrow sectors. This is an important property that will allow us to recover the original theory from the extended theory.
 
For $0\leq k_1,\dots,k_n \leq d-1$, let $\pi:\text{\calligra{C}}\longrightarrow \Wmod$ be the universal curve and $\text{\calligra{L}}\longrightarrow \text{\calligra{C }}$, the universal $d$-spin structure. Let $\mathcal{D}_i\subset\text{\calligra{C }}$ denote the divisor of the $i$-th marking. We define the {\it extended universal $d$-spin structure} by
\[
\Wext :=\Luni \text{   }\otimes \mathcal{O}_{\text{\calligra{C }}}\left(-\sum_{i=0}^n \mathcal{D}_i\right).
\]
From this definition, it is clear that 
\[
\text{age}_{\mathcal{D}_i}\left(\Wext\right) =\left\langle \frac{k_i+1}{d}-\frac{1}{d} \right\rangle=\frac{k_i}{d}.
\]
Consider the forgetful morphism $\rho:\text{\calligra{C}}\longrightarrow \Ccoarse$ that forgets the stack-theoretic structure along the marking divisors $\mathcal{D}_1,\dots, \mathcal{D}_n$ but not along the nodes, and the coarse projection morphism $|\pi |:\Ccoarse\longrightarrow \left | \Wmod \text{ }\right |$. We then have the following isomorphism of orbifold bundles
\begin{equation}\label{Wext-root}
\Wext^{\otimes d}\cong \rho^{\ast}\omega_{|\pi |},
\end{equation}
where $\omega_{|\pi |}$ is the relative dualizing sheaf of the coarse projection $|\pi|$. We define the {\it extended obstruction bundle} to be
\[
-\mathbb{R}\pi_{\ast}\left( \bigoplus_{j=1}^N \Wext^{\otimes w_j}\right).
\]

The following lemma ensures that the extended obstruction bundle is an orbifold vector bundle:
\begin{lemma}
Let $\mathcal{C}$ be a fiber of $\pi:\text{\calligra{C}}\longrightarrow \Wmod$. Then, 
\[
H^0\left (\mathcal{C},\left. \Wext^{\otimes w_j}\right |_{\mathcal{C}}\right)=0,
\]
for $j=1,\dots, N$.
\end{lemma}
\begin{proof}
First, note that
\[
H^0(\mathcal{C},\left. \rho^{\ast} \omega_{|\pi |}\right |_{\mathcal{C}})=H^0(|\mathcal{C}|,\left.  \omega_{|\pi |}\right |_{|\mathcal{C}|})=H^0(|\mathcal{C}|,\omega_{|\mathcal{C}|})=0
\]
since $|\mathcal{C}|$ is a curve of genus zero. From Equation (\ref{Wext-root}) and the Gorenstein condition, it follows that $\Wext^{\otimes w_j}$ is a root of $\rho^{\ast}\omega_{|\pi |}$, establishing the claim.
\end{proof}
The previous lemma implies that the extended obstruction bundle is given by
\[
R^1 \pi_{\ast}\left( \bigoplus_{j=1}^N \Wext^{\otimes w_j}\right),
\]
and we can now define extended FJRW invariants in terms of this orbifold vector bundle.
\begin{definition}
Let $\phi_{k_i}\in H_{\text{FJRW}}^{ext}(W)$ for $i=1,\dots, n$. Define {\it extended} FJRW {\it invariants} to be
\[
\langle \tau_{a_1}(\phi_{k_1}),\dots,\tau_{a_n}(\phi_{k_n})\rangle_{0,n}^{ext}:=\int_{[\Wmod]}\prod_{i=1}^n \widetilde{\psi}_i^{a_i}\cup\prod_{j=1}^N c_{top}\left(\left( R^1\pi_{\ast} \Wext^{\otimes w_j} \right)^{\ast}\right).
\]
\end{definition}
The following lemma allows us to recover the narrow part of FJRW theory in terms of the extended invariants. A proof can be found in \cite[Proposition 3.2.]{CIR}.
\begin{lemma}
If $k_i\in\bf{Nar}$ for all $i=1,\dots,n$, then
\[
\langle \tau_{a_1}(\phi_{k_1}),\dots,\tau_{a_n}(\phi_{k_n})\rangle_{0,n}^{ext}=\langle \tau_{a_1}(\phi_{k_1}),\dots,\tau_{a_n}(\phi_{k_n})\rangle_{0,n}^{\text{FJRW}}.
\]
Otherwise, the invariants vanish.
\end{lemma}



\subsection{Givental's formalism and FJRW big $I$-function}
In the Givental framework (see, for example, \cite{Coates-Givental}), the information of the genus zero invariants is contained in a Lagrangian cone inside a symplectic space. As in the case of Gromov-Witten theory, FJRW theory also fits the picture provided by the formalism of Givental. In this section we review this construction. We refer the reader to \cite{CPS} for an accessible introduction to the subject.

\subsubsection{The symplectic space}

Let $H$ denote $H_{\text{CR}}^{amb}(X_{W};\C)$, $H_{\text{FJRW}}^{nar}(W)$ or $H_{\text{FJRW}}^{ext}(W)$. Choose a homogeneous basis $\{\phi_0,\dots,\phi_s\}$ for $H$ such that $\phi_0=\textbf{1}$ is the identity element of $H$. Define a pairing matrix by $g_{ij}:=(\phi_i,\phi_j)$, where $(\text{ },\text{ })$ is either the pairing of CR-cohomology or the pairing of FJRW theory. Let $g^{ij}$ be its inverse, we then have a dual basis $\{\phi^0,\dots,\phi^s\}$ defined as $\phi^i:=\sum_j g^{ij}\phi_j$, for $i=0,\dots, s$.

The Givental symplectic space is defined as $\mathcal{V}:=H\otimes\C ((z^{-1}))$ and it comes equipped with a symplectic form $\Omega$ defined as
\[
\Omega (\textbf{f},\textbf{g}):= \underset{z=0}{\text{Res}}(\textbf{f}(-z),\textbf{g}(z))
\]

where $\textbf{f}(z)$ and $\textbf{g}(z)$ are elements of $\mathcal{V}$.

The symplectic form $\Omega$ induces a Lagrangian polarization $\mathcal{V}=\mathcal{V}_{+}\oplus\mathcal{V}_{-}$ on the symplectic space, with $\mathcal{V}_{+}:=H[z]$ and $\mathcal{V}_{-}:=z^{-1}H  [[z^{-1}]]$. Note that we have an identification $\mathcal{V}\cong T^{\ast}\mathcal{V}_{+}$. Under this polarization a typical element $\textbf{q}+\textbf{p}$ of $\mathcal{V}$ can be written in Darboux coordinates as
\[
\sum_{k\geq 0}\sum_{\alpha =0}^s q_{k}^{\alpha}\phi_{\alpha}z^k + \sum_{l \geq 0}\sum_{\beta =0}^s p_{l,\beta}\phi^{\beta}(-z)^{-l-1}
\]

\subsubsection{The potentials}

The invariants in both theories can be succinctly packaged into generating functions known as potentials. The genus $g$ potential is defined as
\begin{equation}
\mathcal{F}_{\text{GW}}^g(\mathbf{t}):=\sum_{\substack{d\geq 0\\ n\geq 0}}\text{ }\sum_{\substack{a_1,\dots,a_n\\ h_1,\dots,h_n}} \langle \tau_{a_1}(\phi_{h_1}),\dots, \tau_{a_n}(\phi_{h_n})\rangle_{g,n,d}^{\text{GW}} \frac{t_{a_1}^{h_1}\dots t_{a_n}^{h_n}}{n!}Q^d
\end{equation}
in the case of Gromov-Witten theory and as
\begin{equation}\label{genusg-potential}
\mathcal{F}_{\text{FJRW}}^g(\mathbf{t}):=\sum_{ n\geq 0}\text{ }\sum_{\substack{a_1,\dots,a_n\\ h_1,\dots,h_n}} \langle \tau_{a_1}(\phi_{h_1}),\dots, \tau_{a_n}(\phi_{h_n})\rangle_{g,n}^{\text{FJRW}} \frac{t_{a_1}^{h_1}\dots t_{a_n}^{h_n}}{n!}
\end{equation}
in the case of FJRW theory.
Lastly, we also have a total descendent potential which encodes invariants of all genera and is defined as
\begin{equation}\label{descendent}
\mathcal{D}:=\text{exp} \biggl ( \sum_{g\geq 0} \hbar^{g-1}\mathcal{F}^g \biggr )
\end{equation}
for both theories.

\subsubsection{The Lagrangian cone}

After performing the \textit{Dilaton shift}
\[
q_a^{h}=
\begin{cases}
t_1^0-1 &\text{if } (a,h)=(1,0)\\
t_a^h &\text{otherwise}
\end{cases}
\]
the genus zero potential $\mathcal{F}^0$ can be considered as a power series in the Darboux coordinates $q_a^h$. The Lagrangian cone $\mathcal{L}$ is defined as
\begin{equation}\label{cone}
\mathcal{L}:=\left\{(\textbf{q},\textbf{p})\in T^{\ast}\mathcal{V}_{+}\text{ }\middle |\text{ }\textbf{p}=\text{d}_{\textbf{q}}\mathcal{F}^0\right\}\subset \mathcal{V}.
\end{equation}

From Theorem 1 in \cite{Givental} it follows that this Lagrangian submanifold is a Lagrangian cone with its vertex at the origin such that
\begin{equation}\label{lagrangian}
zT_{\textbf{f}}\mathcal{L}=\mathcal{L}\cap T_{\textbf{f}}\mathcal{L},\quad \text{for all $\textbf{f}\in\mathcal{L}$}
\end{equation}
if and only if $\mathcal{D}$ satisfies the Topological Recursion Relations (TRR), the String Equation (SE) and the Dilation Equation (DE). In Gromov-Witten theory $\mathcal{D}_{\text{GW}}$ is known to satisfy these three conditions. In FJRW theory the same is guaranteed as a consequence of Theorem 4.2.9 in \cite{FJRa}.

We define the $J$-function in Gromov-Witten theory as
\begin{equation}
J_{\text{GW}}(\textbf{t},z):=z\phi_0+\sum_{h} t_0^h\phi_h+\sum_{\substack{n\geq 0\\ k\geq 0}}\text{ }\sum_{h_1,\dots,h_n}\sum_{\epsilon , d}\frac{t_{0}^{h_1}\dots t_{0}^{h_n}}{n!z^{k+1}} \langle \phi_{h_1},\dots, \phi_{h_n},\tau_k(\phi_{\epsilon})\rangle_{0,n+1,d}^{\text{GW}} \phi^{\epsilon}Q^d
\end{equation}
and in FJRW theory as
\begin{equation}
J_{\text{FJRW}}(\textbf{t},z):=z\phi_0+\sum_{h} t_0^h\phi_h+\sum_{\substack{n\geq 0\\ k\geq 0}}\text{ }\sum_{h_1,\dots,h_n}\sum_{\epsilon}\frac{t_{0}^{h_1}\dots t_{0}^{h_n}}{n!z^{k+1}} \langle \phi_{h_1},\dots, \phi_{h_n},\tau_k(\phi_{\epsilon})\rangle_{0,n+1}^{\text{FJRW}} \phi^{\epsilon}.
\end{equation}

Note that $J(\textbf{t},-z)$ is the intersection of the Lagrangian cone with the slice $-z\phi_0+\textbf{t}+\mathcal{O}(z^{-1})\subset \mathcal{V}$. It is a well-known consequence of Equation (\ref{lagrangian}) that $J(\textbf{t},-z)$ generates the entire Lagrangian cone and therefore it contains all the information of the genus zero theory.
\subsubsection{The $e_{\C^{\ast}}$-twisted theory}
Consider the diagonal action of $\C^{\ast}$ on the extended obstruction bundle $\bigoplus_{j=1}^{N}R^1\pi_{\ast}\Wext ^{\otimes w_j}$, which scales the fibers and acts trivially on the base. Let $H_{\C^\ast}^{\ast}(\text{pt})=\C[\lambda]$ and define $R:=\C[\lambda][[s_0,s_1,\dots]]$. The state space of the twisted FJRW theory is defined to be
\[
\mathcal{V}^{tw,\bf{s}}:=H_{\text{FJRW}}^{ext}(W)\otimes R\otimes\C((z^{-1})).
\]
We define a non-degenerate pairing on $\mathcal{V}^{tw,\bf{s}}$ by
\begin{equation}\label{twisted-pairing}
(\phi_i,\phi_k)^{tw,\bf{s}}:=
\begin{cases}
\prod_{j: (i+1)w_j\in d\mathbb{Z}}\exp (-s_0) &\text{if } i+k+2\equiv 0 \text{ (mod d)}\\
0 &\text{otherwise}.
\end{cases}
\end{equation}
Then, $\mathcal{V}^{tw,\bf{s}}$  becomes a symplectic space by defining the form
\[
\Omega (\textbf{f},\textbf{g})^{tw,\bf{s}}:= \underset{z=0}{\text{Res}}(\textbf{f}(-z),\textbf{g}(z))^{tw,\bf{s}}
\]
where $\textbf{f}(z)$ and $\textbf{g}(z)$ are elements of $\mathcal{V}^{tw,\bf{s}}$.

Given a K-class $[V]\in K^0(\Wmod)$, we define its {\it universal characteristic class} as
\begin{equation}\label{characteristic-class}
c_{\bf{s}}\left( [V] \right):=\exp \left( \sum_{l \geq 0} s_l \text{ch}_l ([V]) \right).
\end{equation}
We define \textit{twisted FJRW invariants} by
\[
\langle \tau_{a_1}(\phi_{k_1}),\dots,\tau_{a_n}(\phi_{k_n})\rangle_{0,n}^{tw,\bf{s}}:=\int_{[\Wmod]}\prod_{i=1}^n \widetilde{\psi}_i^{a_i}\cup c_{\bf{s}}\left(-\bigoplus_{j=1}^NR^1\pi_{\ast} \Wext^{\otimes w_j} \right),
\]
i.e. we are twisting the theory by the extended obstruction bundle. We can also define potentials $\mathcal{F}_{\text{FJRW}}^{g,tw,\bf{s}}(\bf{t})$ and $\mathcal{D}^{tw,\bf{s}}$, as in Equations (\ref{genusg-potential}) and (\ref{descendent}) respectively. One can check that the twisted invariants satisfy TRR, SE, and DE (see for example \cite{Chiodo-Ruan2}). Therefore, the {\it twisted Lagrangian cone} 
\[
\mathcal{L}^{tw,\bf{s}}:=\{ (\bf{q},\bf{p}) \mid \bf{p}=\text{d}_{\bf{q}} \mathcal{F}_{\text{FJRW}}^{0,tw,\bf{s}} \}
\]
satisfies
\[
zT_{\textbf{f}}\mathcal{L}^{tw,\bf{s}}=\mathcal{L}^{tw,\bf{s}}\cap T_{\textbf{f}}\mathcal{L}^{tw,\bf{s}},\quad \text{for all $\textbf{f}\in\mathcal{L}^{tw,\bf{s}}$}.
\]
A key property of the twisted Lagrangian cone is the fact that under a certain specialization of the parameters $s_0,s_1,\dots$, we can recover the standard FJRW Lagrangian cone of Equation (\ref{cone}), after taking the non-equivariant limit $\lambda\rightarrow 0$. To verify this claim, define
\begin{equation}\label{parameters}
s_l:=
\begin{cases}
-\log (\lambda) &\text{if } l=0\\
\frac{(l-1)!}{\lambda^l} &\text{if $l>0$}.
\end{cases}
\end{equation}
We then have the following lemma:
\begin{lemma}
Let $s_0,s_1,\dots $ be defined as in Equation $(\ref{parameters})$. Then,
\[
c_{\bf{s}}\left(-\bigoplus_{j=1}^NR^1\pi_{\ast} \Wext^{\otimes w_j} \right)=\prod_{j=1}^N e_{\C^{\ast}}\left(\left( R^1\pi_{\ast} \Wext^{\otimes w_j} \right)^{\ast}\right),
\]
where $e_{\C^{\ast}}$ is the equivariant Euler class.
\end{lemma}
\begin{proof}
To prove the lemma, we make use of the following identity. Given a K-class $[V]$, its equivariant Euler class can be written in terms of the non-equivariant Chern character as follows:
\begin{equation}\label{eulerclass}
e_{\C^{\ast}}([V])=\exp\left(\log (\lambda)\text{ch}_0([V]) +\sum_{l>0} (-1)^{l-1}\frac{(l-1)!}{\lambda^l}\text{ch}_l([V])\right).
\end{equation}
Thus, if we specialize the parameters $s_0,s_1$ to Equation (\ref{parameters}), we obtain
\[
\begin{split}
c_{\bf{s}}\left(-R^1\pi_{\ast} \Wext^{\otimes w_j} \right)&=\exp \left( \sum_{l \geq 0} s_l \text{ch}_l \left(-R^1\pi_{\ast} \Wext^{\otimes w_j}\right) \right)\\
&=\exp \left( \sum_{l \geq 0} (-1)^{l-1}s_l \text{ch}_l\left( \left(R^1\pi_{\ast} \Wext^{\otimes w_j}\right)^{\ast}\right) \right)\\
&=\exp \left( \log (\lambda)\text{ch}_0\left( \left(R^1\pi_{\ast} \Wext^{\otimes w_j}\right)^{\ast}\right) \right.\\
 &\quad\quad\quad+\left. \sum_{l >0} (-1)^{l-1}\frac{(l-1)!}{\lambda^l} \text{ch}_l\left( \left(R^1\pi_{\ast} \Wext^{\otimes w_j}\right)^{\ast}\right) \right)\\
 &=e_{\C^{\ast}}\left(\left( R^1\pi_{\ast} \Wext^{\otimes w_j} \right)^{\ast}\right)\quad\quad \text{by Equation (\ref{eulerclass})}.
\end{split}
\]
The desired result follows from this.
\end{proof}    
\subsubsection{The untwisted theory} By specializing the parameters $s_0,s_1,\dots$ of the twisted theory to $s_l=0$ for $l\geq 0$, Equation (\ref{characteristic-class}) becomes $c_{\bf{s}}([V])=1$. The invariants obtained by this specialization are known as {\it untwisted invariants} and can be written as
\begin{equation}\label{untwisted-invariants}
\langle \tau_{a_1}(\phi_{k_1}),\dots,\tau_{a_n}(\phi_{k_n})\rangle_{0,n}^{un}:=d\int_{[\Wmod]}\prod_{i=1}^n \widetilde{\psi}_i^{a_i}.
\end{equation}
The untwisted invariants can be easily computed via push-forward along the forgetting morphism $ \text{st}:\Wmod \longrightarrow\overline{\mathcal{M}}_{0,n}$. We thus obtain
\begin{equation}\label{hodge-integrals}
\begin{split}
\langle \tau_{a_1}(\phi_{k_1}),\dots,\tau_{a_n}(\phi_{k_n})\rangle_{0,n}^{un}&=\int_{\overline{\mathcal{M}}_{0,n}}\prod_{i=1}^n \psi_i^{a_i}\\
&=\frac{\sum_{i=1}^n a_i}{a_1!\dots a_n!}
\end{split}
\end{equation}
as soon as $\sum_{i=1}^n a_i=n-3$ and $2+\sum_{i=1}^n k_i\in d\mathbb{Z}$, and zero otherwise. Note that the factor of $d$ in Equation (\ref{untwisted-invariants}) was canceled because of the relation $\text{st}_{\ast}[\Wmod]=\frac{1}{d}[{\overline{\mathcal{M}}_{0,n}}]$.

We can now define an untwisted symplectic space $(\mathcal{V}^{un},\Omega^{un})$, as well as potentials $\mathcal{F}_{\text{FJRW}}^{g,un}(\bf{t})$, $\mathcal{D}^{un}$, and a Lagrangian cone $\mathcal{L}^{un}$. Since the untwisted invariants are equivalent to Hodge-type integrals on ${\overline{\mathcal{M}}_{0,n}}$, it follows that $\mathcal{D}^{un}$ satisfies (TRR), (SE), and (DE). Therefore, the untwisted Lagrangian cone satisfies 
\[
zT_{\textbf{f}}\mathcal{L}^{un}=\mathcal{L}^{un}\cap T_{\textbf{f}}\mathcal{L}^{un},\quad \text{for all $\textbf{f}\in\mathcal{L}^{un}$}.
\]
\subsubsection{Computing the FJRW big I-function}
We now compute the so-called big $I$-function for FJRW theory. This function is a parametric family that lies on the Lagrangian cone $\mathcal{L}_{\text{FJRW}}$ and from which the FJRW $J$-function can be completely determined by means of a mirror theorem. In order to determine the $I$-function, we must proceed in several steps, which we now outline:
\begin{enumerate}
\item We find the untwisted $J$-function $J^{un}(\textbf{t},-z)$ that lies on the cone $\mathcal{L}^{un}$.\\
\item Apply the transformation $\exp \biggl (-\sum_{j=1}^N G_0(zq_j\nabla +zq_j,z)\biggr )$ (defined below) to the untwisted J-function $J^{un}(\textbf{t},-z)$, to obtain a new family on $\mathcal{L}^{un}$.\\
\item Define a symplectic transformation $\bigtriangleup : (\mathcal{V}^{un},\Omega^{un})\rightarrow (\mathcal{V}^{tw,\textbf{s}},\Omega^{\textbf{s}})$  satisfying ${\mathcal{L}}^{tw,\textbf{s}}=\bigtriangleup ({\mathcal{L}}^{un})$.\\
\item Apply the transformation $\bigtriangleup$ to the family $\textbf{t}\mapsto \exp \biggl (-\sum_{j=1}^N G_0(zq_j\nabla +zq_j,z)\biggr )J^{un}(\textbf{t},-z)$, to obtain a family lying on 
${\mathcal{L}}^{tw ,\textbf{s}}$.\\
\item Set the parameters $s_0,s_1,\dots$ equal to Equation (\ref{parameters}) and take the non-equivariant limit $\lambda\rightarrow 0$, to obtain a family $I_{\text{FJRW}}(\textbf{t},-z)$ lying on $\mathcal{L}_{\text{FJRW}}$.
\end{enumerate}

Before moving forward with the construction outlined above, we need to set some notation. For a sequence of parameters $s_0,s_1,s_2,\dots$ we define 
\[
\textbf{s}(x):=\sum_{k\geq 0} s_k\frac{x^k}{k!}.
\]
Also, recall that the Bernoulli polynomials $B_{n}(x)$ are defined by the relation
\[
\sum_{n=0}^{\infty}B_{n}(x)\frac{z^n}{n!}=\frac{z e^{zx}}{e^z-1}.
\]
We define a function $G_y(x,z)$ in terms of the Bernoulli polynomials as
\[
G_y(x,z):=\sum_{l,m\geq 0} s_{l+m-1} \frac{B_m (y)}{m!}\frac{x^l}{l!}z^{m-1},
\]
where $s_{-1}:=0$.
This function satisfies the following relations:
\begin{eqnarray}\label{G-function}
G_y(x,z)&=G_0(x+yz,z),\\
G_0(x+z,z)&=G_0(x,z)+\textbf{s}(x).
\end{eqnarray}

For a vector of non-negative integers $\mathbf{k}=(k_0,\dots,k_{d-1})$, define
\begin{eqnarray}
|\textbf{k}|:=&\sum_{i=0}^{d-1}k_i,\\
h(\textbf{k}):=& \sum_{i=0}^{d-1}ik_i,\\
\widetilde{h}(\textbf{k}):=&d \left\langle \frac{\sum_{i=0}^{d-1}ik_i}{d} \right\rangle,
\end{eqnarray}
i.e., $\widetilde{h}(\textbf{k})\equiv \sum_{i=0}^{d-1}ik_i\text{ }(\text{mod d})$ and $\widetilde{h}(\textbf{k})\in \{0,\dots, d-1\}$.

The following lemma provides step (1) in the construction of the FJRW I-function.
\begin{lemma}
The untwisted $J$-function is given by
\[
J^{\text{un}}(\textbf{t},z)=\sum_{\textbf{k}=(k_0,\dots, k_{d-1})\in\mathbb{Z}_{\geq 0}^d} J_{\textbf{k}}^{\text{un}}(\textbf{t},z),
\]
where
\[
J_{\textbf{k}}^{\text{un}}(\textbf{t},z):=\frac{1}{z^{|\textbf{k}|-1}}\frac{(t_0^0)^{k_0}\dots (t_0^{d-1})^{k_{d-1}}}{k_0!\dots k_{d-1}! }\phi_{\widetilde{h}(\textbf{k})}.
\]
\end{lemma}
\begin{proof}
The untwisted J-function was defined as
\[
\begin{split}
J^{un}(\textbf{t},-z)=&-z\phi_0 +\sum_{i=0}^{d-1}t_0^i\phi_i\\
&+\sum_{\substack{n\geq 0\\ l\geq 0}}\text{ }\sum_{0\leq h_{1},\dots ,h_{n}\leq d-1} \sum_{\epsilon=0}^{d-1}\frac{t_0^{h_1}\dots t_0^{h_n}}{n!(-z)^{l+1}}\langle \phi_{h_1},\dots, \phi_{h_n},\tau_{l}(\phi_{\epsilon})\rangle_{0,n+1}^{un}\phi^{\epsilon}.
\end{split}
\]
For dimensional reasons, an invariant appearing in the formula for the J-function will vanish unless $n+1-3=l$. We also need $2+\epsilon+\sum_{i=0}^n h_i\in d\mathbb{Z}$, or the invariants will vanish as a consequence of Lemma \ref{moduli-lemma}. Rewrite the invariant $\langle \phi_{h_1},\dots, \phi_{h_n},\tau_{l}(\phi_{\epsilon})\rangle_{0,n+1}^{un}$ as
\[
\langle \underbrace{\phi_0,\dots, \phi_0}_\text{$k_0$},\underbrace{\phi_1,\dots ,\phi_1}_\text{$k_1$},\dots, \underbrace{\phi_{d-1},\dots,\phi_{d-1}}_\text{$k_{d-1}$},\tau_{l}(\phi_{\epsilon})\rangle_{0,n+1}^{un},
\]
where $\phi_i$ appears $k_i$ times, $i=0,\dots,d-1$. This requires that $n=\sum_{i=0}^{d-1}k_i=|\textbf{k}|$ and that $\sum_{i=1}^nh_i=\sum_{i=0}^{d-1}ik_i$. Therefore, we need $2+\epsilon +\widetilde{h}(\textbf{k})\in d\mathbb{Z}$.

If these conditions are met, it follows from Equation (\ref{hodge-integrals}) that the above invariant is equal to $1$. Note that each invariant of this form appears $(k_0+\dots +k_{d-1})!/(k_0!\dots k_{d-1}!)$ times in the J-function. Thus,
\[
\begin{split}
J^{un}(\textbf{t},-z)=&-z\phi_0 +\sum_{i=0}^{d-1}t_0^i\phi_i\\
&+\sum_{\substack{\textbf{k}=(k_0,\dots , k_{d-1})\\ :|\textbf{k}|\geq 2}}\frac{(t_0^{0})^{k_0}\dots (t_0^{d-1})^{k_{d-1}}}{|\textbf{k}|!(-z)^{|\textbf{k}|-1}}\frac{(k_0+\dots +k_{d-1})!}{k_0!\dots k_{d-1}!} \phi_{\widetilde{h}(\textbf{k})}\\
=&\sum_{\textbf{k}=(k_0,\dots , k_{d-1})\in\mathbb{Z}_{\geq 0}^{d}}\frac{(t_0^{0})^{k_0}\dots (t_0^{d-1})^{k_{d-1}}}{(-z)^{|\textbf{k}|-1}k_0!\dots k_{d-1}!} \phi_{\widetilde{h}(\textbf{k})},
\end{split}
\]
where we have used $\phi^{\epsilon}=\phi_{\widetilde{h}(\textbf{k})}$ (this follows from Equation (\ref{twisted-pairing}) and the condition $2+\epsilon +\widetilde{h}(\textbf{k})\in d\mathbb{Z}$).
\end{proof}
The second step in the construction of the FJRW I-function consists of using the untwisted J-function to construct a new family on the cone $\mathcal{L}^{un}$. The following lemma provides the realization of this step.
\begin{lemma}\label{family}
The family
\begin{equation}\label{family-definition}
\textbf{t}\mapsto \exp \biggl (-\sum_{j=1}^N G_0(zq_j\nabla +zq_j,z)\biggr )J^{\text{un}}(\textbf{t},-z)
\end{equation}
lies on the untwisted Lagrangian cone $\mathcal{L}^{\text{un}}$.
\end{lemma}
\begin{proof}
See \cite[Lemma 4.1.10]{Chiodo-Ruan2}.
\end{proof}

Step (3) in the construction of the FJRW I-function consists of defining a symplectic transformation that maps the untwisted Lagrangian cone onto the twisted cone. The following theorem defines this transformation.

\begin{theo}\label{transformation}
Define the linear symplectic transformation $\bigtriangleup : (\mathcal{V}^{\text{un}},\Omega^{\text{un}})\rightarrow (\mathcal{V}^{\text{tw},\textbf{s}},\Omega^{\textbf{s}})$ by
\begin{equation}
\bigtriangleup :=\bigoplus_{i=0}^d\exp\left (  \sum_{j=1}^N\sum_{l\geq 0} \frac{s_l B_{l+1}\left(\left\langle \frac{iw_j}{d}\right \rangle+\frac{w_j}{d}\right )z^l}{(l+1)!}\right).
\end{equation}
Then, ${\mathcal{L}}^{\text{tw,\textbf{s}}}=\bigtriangleup ({\mathcal{L}}^{\text{un}})$.
\end{theo}
\begin{proof}
The proof of \cite[Proposition 4.1.5]{Chiodo-Ruan2} extends word for word.
\end{proof}

For the fourth step of the construction, we define the twisted FJRW $I$-function and show that this function is obtained by applying the symplectic transformation $\bigtriangleup$ to the family defined in Equation (\ref{family-definition}).
\begin{definition}
Define the twisted FJRW $I$-function by
\begin{equation}\label{twistedI}
I^{\text{tw},\textbf{s}}(\textbf{t},z):=\sum_{\textbf{k}\in\mathbb{Z}_{\geq 0}^d} M_{\textbf{k}}(z)J_{\textbf{k}}^{\text{un}}(\textbf{t},z),
\end{equation}
where $M_{\textbf{k}}(z):=\prod_{j=1}^N \exp \biggl ( -\sum_{0\leq m <\lfloor q_jh(\textbf{k})\rfloor} \textbf{s}(-(q_j+\langle q_jh(\textbf{k})\rangle+m)z) \biggr )$.
\end{definition}
\begin{theo}
The family
\[
\textbf{t}\mapsto I^{\text{tw},\textbf{s}}(\textbf{t},-z),
\]
defined in Equation $(\ref{twistedI})$, lies on the twisted Lagrangian cone $\mathcal{L}^{\text{tw},\textbf{s}}$.
\end{theo}
\begin{proof}
This proof only requires a straightforward computation similar to the one found in \cite{Chiodo-Ruan2} Theorem 4.1.6. We repeat it here for the reader's convenience. From Theorem \ref{transformation} and Lemma \ref{family}, it follows that the family
\[
\textbf{t}\mapsto \bigtriangleup\exp \biggl (-\sum_{j=1}^N G_0(zq_j\nabla +zq_j,z)\biggr )J^{\text{un}}(\textbf{t},-z)
\]
lies on the twisted cone $\mathcal{L}^{\text{tw},\textbf{s}}$. Using the definition of $G_{y}(x,z)$ and Equation (\ref{G-function}), the transformation $\bigtriangleup$ may be rewritten as
\[
\bigtriangleup=\bigoplus_{i=0}^{d-1}\exp\biggl (  \sum_{j=1}^N G_0(\langle iq_j\rangle z +q_jz,z)\biggr ).
\]
The relevant family can now be written as
\[
\begin{split}
\bigtriangleup &\exp \biggl (-\sum_{j=1}^N G_0(zq_j\nabla +zq_j,z)\biggr )J^{\text{un}}(\textbf{t},-z)\\
&= \sum_{\textbf{k}} \exp \left (\sum_{j=1}^N\left\{ G_0( \langle \sum_{i=0}^{d-1}i k_iq_j \rangle z + q_jz,z) -G_0( \sum_{i=0}^{d-1}i k_iq_jz  + q_jz,z)\right \}\right )J_{\textbf{k}}^{\text{un}}(\textbf{t},-z)\\
&= \sum_{\textbf{k}} \exp \left (\sum_{j=1}^N \sum_{m=0}^{\left\lfloor \sum_{i=0}^{d-1} ik_iq_j\right\rfloor-1}\textbf{s}\left ( \left\langle \sum_{i=0}^{d-1} i k_iq_j \right\rangle  z + q_jz +mz\right) \right )J_{\textbf{k}}^{\text{un}}(\textbf{t},-z)\\
&= \sum_{\textbf{k}} \prod_{j=1}^N\exp \left (\sum_{m=0}^{\left\lfloor q_j h(\textbf{k})\right\rfloor -1}\textbf{s}\left ( \left\langle  q_jh(\textbf{k}) \right\rangle  z + q_jz +mz\right) \right )J_{\textbf{k}}^{\text{un}}(\textbf{t},-z)\\
&= \sum_{\textbf{k}} M_{\textbf{k}}(-z)J_{\textbf{k}}^{\text{un}}(\textbf{t},-z)\\
&=I^{\text{tw},\textbf{s}}(\textbf{t},-z)
\end{split}
\]
where in the third line we used the identity $\sum_{m=0}^{l-1}\textbf{s}(x+mz)=G_0(x+lz,z)-G_0(x,z)$.
\end{proof}
The last step in the construction of the I-function is to specialize the parameters $s_0,s_1,\dots$ to the values of Equation (\ref{parameters}). By doing  this, the hypergeometric modification factor $M_{\bf{k}}(-z)$ becomes
\[
\begin{split}
M_{\bf{k}}(-z)&=\prod_{j=1}^N \prod_{m=0}^{\lfloor q_jh(\textbf{k})\rfloor -1}\exp \biggl (  -\textbf{s}((q_j+\langle q_jh(\textbf{k})\rangle+m)z) \biggr )\\
&=\prod_{j=1}^N \prod_{m=0}^{\lfloor q_jh(\textbf{k})\rfloor -1}\exp \left ( -\sum_{l\geq 0} s_l \frac{((q_j+\langle q_jh(\textbf{k})\rangle+m)z)^l}{l!} \right )\\
&=\prod_{j=1}^N \prod_{m=0}^{\lfloor q_jh(\textbf{k})\rfloor -1}\exp \left ( \log (\lambda)-\sum_{l >0} s_l \frac{((q_j+\langle q_jh(\textbf{k})\rangle+m)z)^l}{l!} \right)\\
&=\prod_{j=1}^N \prod_{m=0}^{\lfloor q_jh(\textbf{k})\rfloor -1}\exp \left ( \log (\lambda)-\sum_{l >0}\frac{1}{l} \left(\frac{(q_j+\langle q_jh(\textbf{k})\rangle+m)z}{\lambda}\right)^l \right )\\
&=\prod_{j=1}^N \prod_{m=0}^{\lfloor q_jh(\textbf{k})\rfloor -1}\exp \left ( \log (\lambda)+\log\left(1-\frac{(q_j+\langle q_jh(\textbf{k})\rangle+m)z}{\lambda}\right) \right )\\
&=\prod_{j=1}^N \prod_{m=0}^{\lfloor q_jh(\textbf{k})\rfloor -1}  \left(\lambda-(q_j+\langle q_jh(\textbf{k})\rangle+m)z\right), 
\end{split}
\]
where we have used the Taylor expansion for $\log (1-x)$ around $x=0$. We thus obtain the following specialization of the twisted I-function:
\begin{equation}\label{twisted-I-function}
I^{tw}(\textbf{t},z;\lambda):=z\sum_{\textbf{k}=(k_0,\dots , k_{d-1})\in\mathbb{Z}_{\geq 0}^{d}}\prod_{i=0}^{d-1} \frac{(t_0^{i})^{k_i}}{z^{k_i} k_i!} \prod_{j=1}^N \prod_{\substack{0 \leq b <q_jh(\textbf{k})\\ \langle b\rangle = \langle q_jh(\textbf{k})\rangle}} \left(\lambda+(q_j +m)z\right) \phi_{\widetilde{h}(\textbf{k})}
\end{equation}
\begin{theo}\label{twisted-mirror}
Let $s_0,s_1,\dots$ be given by Equation $(\ref{parameters})$. Then, the twisted $J$-function $J^{tw}(\mathbf{\tau},z;\lambda)$ can be obtained in terms of the twisted $I$-function $I^{tw}(\textbf{t},z;\lambda)$ defined in Equation $(\ref{twisted-I-function})$. More explicitly, we have the following relation:
\begin{equation}\label{mirror-equation2}
J^{tw}(\mathbf{\tau},z;\lambda)=I^{tw}(\textbf{t},z;\lambda)+\sum_{i=0}^{d-1}c_i(t,z)z\frac{\partial I^{tw}}{\partial t_0^i}(\textbf{t},z;\lambda),
\end{equation}
where $c_i(\textbf{t},z)$ is a formal power series in $\textbf{t}$ and $z$, for $i=0,\dots, d-1$, and $\mathbf{\tau}(\textbf{t})$ is determined by the $z^0$-mode of the right-hand-side.
\end{theo}
\begin{proof}
This is a direct consequence of Corollary 5 in \cite{Coates-Givental}.
\end{proof}

To obtain the big $I$-function of $FJRW$ theory, we set $t_i=0$ for $i\notin \textbf{Nar}$, and take the non-equivariant limit $\lambda\rightarrow 0$. Define the big $I$-function to be:
\begin{equation}\label{FJRW-I-function}
I_{\text{FJRW}}(\textbf{t},z):=\lim_{\lambda\rightarrow 0} \left( \left. I^{tw}(\textbf{t},z)\right|_{\substack{t_i=0\\i\notin\textbf{Nar}}}\right).
\end{equation}
We then have the following consequence of the Theorem \ref{twisted-mirror}:
\begin{cor}[Theorem \ref{FJRW-mirror}]
 The FJRW $I$-function defined in Equation $(\ref{FJRW-I-function})$ lies on the FJRW Lagrangian cone $\mathcal{L}_{\text{FJRW}}$. The FJRW $J$-function is completely determined by $I_{\text{FJRW}}$(\textbf{t},z) by means of Equation $(\ref{mirror-equation2})$. 
\end{cor}
\begin{remark}
A similar mirror theorem was obtained by Ross and Ruan in \cite{Ross-Ruan} using wall-crossing in FJRW theory.
\end{remark}
\textit{Example:} Consider the FJRW theory with superpotential $W=x_1^3+x_2^3+x_3^3+x_4^3$ and group of symmetries given by $\langle J_W\rangle \cong \mathbb{Z}_3$. From Equations (\ref{twisted-I-function}) and (\ref{FJRW-I-function}) we can compute the FJRW $I$-function for this theory:
\[
I_{\text{FJRW}}(t_0\phi_0+t_1\phi_1,z)=e^{t_0/z}\left(\sum_{l=0}^{\infty}\frac{t_1^{3l}z^{l+1}}{(3l)!}\frac{\Gamma\left(l+\frac{1}{3}\right)^4}{\Gamma\left(\frac{1}{3}\right)^4}\phi_0 +\sum_{l=0}^{\infty}\frac{t_1^{3l+1}z^l}{(3l+1)!}\frac{\Gamma\left(l+\frac{2}{3}\right)^4}{\Gamma\left(\frac{2}{3}\right)^4}\phi_1\right).
\]  
Its derivatives are
\[
\begin{split}
z\frac{\partial I_{\text{FJRW}}}{\partial t_0}&=e^{t_0/z}\left( \left( z+z^2\frac{t_1^3}{3!}(1/3)^4+\mathcal{O}(z^3)\right) \phi_0+\left( t_1+z\frac{t_1^4}{4!}(2/3)^4+z^2\frac{t_1^7}{7!}(10/9)^4+\mathcal{O}(z^3)\right)\phi_1\right),\\
z\frac{\partial I_{\text{FJRW}}}{\partial t_1}&=e^{t_0/z}\left( \left( \mathcal{O}(z^3)\right) \phi_0+\left(z+ z^2\frac{t_1^3}{3!}(2/3)^4+z^3\frac{t_1^6}{6!}(10/9)^4+\mathcal{O}(z^4)\right)\phi_1\right).
\end{split}
\]
To obtain the $J$-function from Equation (\ref{mirror-equation2}), we must eliminate the positive powers of $z$ degree by degree. A simple computation shows that
\[
\begin{split}
&I_{\text{FJRW}}+\left(-z\frac{t_1^3}{3!}(1/3)^4+\mathcal{O}(z^2)\right)z\frac{\partial I_{\text{FJRW}}}{\partial t_0}+\left(-\frac{t_1^4}{2}(1/3)^4-zt_1^7\left(\frac{34}{7\cdot 9^5}\right)+\mathcal{O}(z^2)\right)z\frac{\partial I_{\text{FJRW}}}{\partial t_1}\\
&=e^{t_0/z}\left( z\phi_0+t_1\phi_1 +\mathcal{O}(z^3)\right).
\end{split}
\]
Therefore, the FJRW $J$-function is simply given by
\[
J_{\text{FJRW}}(t_0\phi_0+t_1\phi_1,z)=e^{t_0/z}(z\phi_0+t_1\phi_1).
\]
From this result we can read off the values of some invariants. For example, we obtain
\[
\begin{split}
&\langle \phi_1,\dots, \phi_1,\tau_k(\phi_{\epsilon})\rangle_{0,n+1}^{\text{FJRW}}=0\quad\text{for any $n\geq2$, $k$, or $\epsilon\in\{0,1\}$,}\\
&\langle \phi_0,\dots, \phi_0,\tau_{n-2}(\phi_1)\rangle_{0,n+1}^{\text{FJRW}}=1\quad\text{for $n\geq2$,}\\
&\langle \phi_0,\dots,\phi_0, \phi_1,\tau_{n-1}(\phi_0)\rangle_{0,n+2}^{\text{FJRW}}=1\quad\text{for $n\geq1$.}
\end{split}
\]
This, in particular, implies that $\phi_1\ast \phi_1=0$ (where the quantum product $\ast$ in FJRW theory is defined in terms of three-point invariants, identically to Gromov-Witten theory). Thus, the quantum ring has nilpotent elements at the Landau-Ginzburg point.
\begin{flushright}
$\square$
\end{flushright}
We end this section with the definition of the so-called small $I$-function. This function is obtained by restricting $I_{\text{FJRW}}(\textbf{t},z)$ to the slice of the Lagrangian cone given by $\textbf{t}=(0,t,0,\dots,0)$. As it will be used in the next section, we explicitly compute this function:
\begin{equation}
\begin{split}
I_{\text{FJRW}}^{small}(t,z):=\lim_{\lambda\rightarrow 0}z\sum_{k=0}^\infty \frac{t^{k+1}}{z^k k!}\prod_{j=1}^N \prod_{m=0}^{\lfloor q_jk \rfloor-1}(\lambda+(q_j+\langle q_jk\rangle+m)z)\phi_{k \text{ (mod d)}}\\
=\lim_{\lambda\rightarrow 0}z\sum_{k=0}^{d-1}\sum_{l=0}^\infty \frac{t^{dl+k+1}}{z^{dl+k}}\frac{\prod_{j=1}^N \prod_{m=0}^{\lfloor q_j(dl+k) \rfloor-1}(\lambda +q_j+\langle q_jk\rangle+m)z}{(dl+k)!} \\ 
\times \left(\prod_{j: (k+1)w_j\in d\mathbb{Z}} \lambda\right) \phi^{d-k-2 \text{ (mod d)}}\\
 =z\sum_{k\in \textbf{Nar}}\sum_{l=0}^\infty \frac{t^{dl+k+1}}{z^{dl+k}}\frac{\prod_{j=1}^N \prod_{m=0}^{\lfloor q_j(dl+k) \rfloor-1}(q_j+\langle q_jk\rangle+m)z}{(dl+k)!}  \phi_k,
\end{split}
\end{equation}

where in the second line we used the twisted pairing defined in Equation (\ref{twisted-pairing}), and in the third line we note that after taking the limit $\lambda\rightarrow 0$, the only terms that survive are those for which $(k+1)w_j\notin d\mathbb{Z}$, for all $j=1,\dots, N$. Thus, we need $k\in \textbf{Nar}$.
\subsubsection{Recovering the big $I$-function from the small $I$-function} Define the \textit{translation operator} $T$, \textit{annihilation operator} $A$, and \textit{gamma class} $\Gamma$ to be
\[
\begin{split}
T\cdot\phi_k:&=\phi_{k+1 \text{(mod d)}}\\
A\cdot\phi_k:&=\begin{cases}
\phi_{k} &\text{if $k\in\textbf{Nar}$}\\
0 &\text{otherwise}
\end{cases}\\
\Gamma\cdot \phi_k:&=\prod_{j=1}^N\Gamma(q_j+\langle q_j k\rangle)\phi_k,
\end{split}
\]
for all $k\in\{0,\dots,d-1\}$. It is not hard to see that the big $I$-function in FJRW-theory can be recovered from the small $I$-function using the following relation:
\begin{multline}\label{big-small}
I_{\text{FJRW}}(\textbf{t},z)=A\cdot\Gamma^{-1}\cdot\sum_{{ \substack{k_i\geq 0 \\ i\in\textbf{Nar}\backslash\{1\} }}}\prod_{i\in\textbf{Nar}\backslash\{1\}}\frac{(t^i)^{k_i}}{z^{k_i}k_i!}\\
\times\prod_{j=1}^N\prod_{l=0}^{\left(q_j\sum_{i\in\textbf{Nar}\backslash\{1\}} ik_i\right) -1}\left(\left(q_j+q_jt\frac{d}{dt}+l\right)z\right)\left(\prod_{i\in\textbf{Nar}\backslash\{1\}}T^{ik_i}\right)\cdot\Gamma\cdot I_{\text{FJRW}}^{small}(t,z)
\end{multline}
\subsubsection{Setting $z=1$ in the $I$-function }\label{section_z=1}
For computational convenience, we will set $z$ equal to $1$ in the rest of the paper. It is important to note that $I_{\text{FJRW}}^{small}(t,z)$ can be uniquely recovered from $I_{\text{FJRW}}^{small}(t,z=1)$ by the following procedure. Define a grading operator $\textbf{Gr}: H_{\text{FJRW}}^{ext}(W)\rightarrow H_{\text{FJRW}}^{ext}(W)$ by
\[
\textbf{Gr}(\phi_k):=\frac{\text{deg}_{\text{FJRW}}(\phi_k)}{2} \phi_k=\sum_{j=1}^N  \left\langle \frac{kw_j}{d}\right\rangle\phi_k.
\]
We then have the following relation:
 \begin{equation}\label{z=1}
 I_{\text{FJRW}}^{small}(t,z)=z^{1-r/d-\textbf{Gr}}I_{\text{FJRW}}^{small}(tz^{r/d},1).
 \end{equation}
\section{Quantum correspondence: matching $I$-functions}\label{correspondence-asymptotic}

\subsection{Irregular Singularities  and Asymptotic Expansions}
Consider the Picard-Fuchs operator defined in Equation (\ref{PF-equation}). After setting $z=1$, we obtain the following differential operator:
\[
\prod_{j=1}^N\prod_{c=0}^{w_j-1}(w_jD_q-c)-q\prod_{c=1}^d(dD_q+c),\quad\text{where $D_q:=q\frac{d}{dq}$.}
\]
In the non-Calabi-Yau setting, this differential operator develops irregular singularities at $q=\infty$ in the Fano case, and at $q=0$ in the general type case. Thus, it is impossible to find holomorphic solutions to the differential equation around these points. It is still possible, however, to construct formal solutions (invoking the theory of formal power series, for example), and to show that these formal solutions arise as the asymptotic expansions of holomorphic solutions at regular singular points. Since the reader may be unfamiliar with these notions, we briefly review some of its basic elements.

\subsubsection{Asymptotic Expansions}
Let $D$ be some sector of the complex plane and let $q_0$ be a point in its closure $\overline{D}$. An \textit{asymptotic sequence as $q\rightarrow q_0$ from $D$} is a collection of functions $\{ \psi_n(q) \}_{n=0}^{\infty}$ defined on $D$ such that $\psi_{n+1}(q)=o(\psi_n(q))$ as $q\rightarrow q_0$ from $D$, for all $n\geq 0$. 

\textit{Example:} It is not hard to see that sequences $\left\{\frac{1}{q^{\lambda}},\frac{1}{q^{1+\lambda}},\frac{1}{q^{2+\lambda}}, \dots\right\}$ and $\left\{\frac{e^{\alpha q}}{q^{\lambda}},\frac{e^{\alpha q}}{q^{1+\lambda}},\frac{e^{\alpha q}}{q^{2+\lambda}}, \dots\right\}$ are examples of asymptotic sequences as $q\rightarrow \infty$. Here $\lambda$ and $\alpha$ are complex numbers. Both of these sequences will be important throughout this section.
\begin{definition}
Given a function $I(q)$ defined on $D$ and an asymptotic sequence $\{\psi(q)_{n}\}_{n=0}^{\infty}$ as $q\rightarrow q_0$ from $D$, we say that the formal sum $\sum_{n=0}^{\infty}a_n\psi_n(q)$ is an \textit{asymptotic expansion} of $I(q)$ as $q\rightarrow q_0$ from $D$ if
\[
\left |I(q)-\sum_{n=0}^m a_n\psi_n(q) \right |=o(\psi_{m}(q))\quad\text{as $q\rightarrow q_0$ from $D$, for all $m\geq 0$},
\] 
and will denote it by
\[
I(q)\sim\sum_{n=0}^{\infty}a_n\psi_n(q)\quad\text{as $q\rightarrow q_0$ from $D$}.
\]
\end{definition}
\textit{Example 1:} Let $D$ be the region in the complex plane defined by $\Re (q)>0$, and consider the asymptotic sequence $\left\{1,\frac{1}{q},\frac{1}{q^2},\dots\right\}$ as $q\rightarrow\infty$. It is not hard to see that $e^{-q}\sim 0$ as $q\rightarrow \infty$ from $D$. Note that even though $e^{-q}$ and $0$ are holomorphic in the whole plane, the asymptotic expansion is only valid in the region D. Also, note that if we choose a different asymptotic sequence, we may obtain a non-zero asymptotic expansion. For example, with the asymptotic sequence $\left\{e^{-q},\frac{e^{-q}}{q},\frac{e^{-q}}{q^2},\dots\right\}$, we obtain, perhaps not surprisingly, that $e^{-q}\sim e^{-q}$ as $q\rightarrow \infty$, in any region of the complex plane.
\begin{flushright}
$\square$
\end{flushright}
\textit{Example 2:} The following example, known as Stirling's formula, gives an asymptotic expansion for the function $\log(\Gamma(q))$ as $q\rightarrow\infty$. Consider the asymptotic sequence $\{q\log(q), q, \log(q),$\\$1, \frac{1}{q}, \frac{1}{q^2},\dots\}$ at $q=\infty$. We then have
\[
\log{\Gamma(q)}\sim (q-\frac{1}{2})\log(q)-q+\frac{1}{2}\log(2\pi)+\sum_{n=1}^{\infty} \frac{B_{2n}}{2n(2n-1)q^{2n-1}}\quad\text{as $q\rightarrow\infty$,}
\]
in the region $\Re(q)>0$. Here $B_n$ are the Bernoulli numbers.
\begin{flushright}
$\square$
\end{flushright}
\textit{Example 3:} For an example of a function with a convergent asymptotic expansion, consider the rational function $\frac{1}{q-1}$. Using the power series expansion for $\frac{1}{1-q}$, one immediately sees that 
\[
\frac{1}{q-1}\sim \sum_{n=1}^{\infty}\frac{1}{q^n}\quad\text{as $q\rightarrow\infty$.}
\]
\begin{flushright}
$\square$
\end{flushright}

The following are important items to keep in mind when working with asymptotic expansions:
\begin{enumerate}
\item In general, asymptotic expansion need not converge. Stirling's formula (Example 2) is an example of a divergent asymptotic expansion.\\
\item Different choices of asymptotic sequences will give rise to different asymptotic expansions. Example 1 illustrates this point.\\
\item Given an asymptotic sequence, if an expansion exists, it is unique.\\
\item In general, a function may have different asymptotic expansions (or no expansion at all) in different regions of the complex plane.\\
\item Different functions may have the same asymptotic expansion.\\
\item Given a formal sum $\sum_{n=0}^{\infty}a_n\psi_n(q)$ at $q_0$ and a region $D$ with $q_0\in\overline{D}$, it is always possible to find an analytic function $f(q)$ defined in $D$ with $f(q)\sim \sum_{n=0}^{\infty}a_n\psi_n(q)$ as $q\rightarrow q_0$ from $D$.
\end{enumerate}
We define the \textit{Laplace transform} of a function $f(\tau)$ to be
\[
(\mathcal{L}f)(u):=\int_0^{\infty} f(\tau)e^{-u\tau}d\tau,
\]
where the integration occurs along some suitable ray. We are interested in finding an asymptotic expansion of $(\mathcal{L}f)(u)$ as $u\rightarrow \infty$. The following result, known as Watson's Lemma, allows us to do exactly that.
\begin{lemma}[Watson's Lemma] Suppose that $f(\tau)$ is absolutely integrable along some ray to infinity:
\[
\int_0^{\infty}|f(\tau)|d\tau <\infty.
\]
Furthermore, assume that $f(\tau)$ is of the form $f(\tau)=\tau^{\lambda}g(\tau)$, where $\Re(\lambda)>-1$, and $g(\tau)$ has continuous derivatives at $\tau=0$ to all orders. Then
\[
(\mathcal{L}f)(u)\sim \sum_{n=0}^{\infty} \frac{g^{(n)}(0)}{n!}\frac{\Gamma(1+\lambda+n)}{u^{n+\lambda+1}}
\]
as $u\rightarrow\infty$ in an appropriate sector of the complex plane (which depends on the ray of integration).

\end{lemma}
\begin{proof}
See, for example, \cite[Sections 2.2 and 2.3]{Miller}, \cite[Chapter 2]{Murray} or \cite{Balser}.
\end{proof}
\subsubsection{Solutions of the Picard-Fuchs Equation at the Landau-Ginzburg point} In this section  we show that the FJRW $I$-function is a solution to the Picard-Fuchs equation at the Landau-Ginzburg point.
\begin{flushleft}\textbf{Fano case:}\end{flushleft} 
Since in this case, the Landau-Ginzburg point is irregular, we can only find formal solutions to the Picard-Fuchs equation. We have the following result in terms of the FJRW $I$-function:
\begin{theo}\label{formal-solutions}
The small FJRW $I$-function encodes formal solutions to the irreducible component of the Picard-Fuchs equation
\begin{equation}\label{picard-fuchs-fano}
\left[\prod_{j=1}^N\prod_{c=0}^{w_j-1}(w_jD_q-c)-q\prod_{c=1}^d(dD_q+c)\right]I_{\text{FJRW}}^{small}(t=q^{-\frac{1}{d}},-1)=0
\end{equation}
at the irregular singular point $q=\infty$, where
\[
I_{\text{FJRW}}^{small}(t=q^{-\frac{1}{d}},-1)=\sum_{k\in \textbf{Nar}} \sum_{l=0}^{\infty} \frac{1}{q^{l+\frac{k+1}{d}}} \frac{(-1)^{dl+k+1}}{(dl+k)!}\prod_{j=1}^N \frac{(-1)^{\lfloor q_j(dl+k)\rfloor}\Gamma\left(q_j(dl+k+1)\right)}{\Gamma \left(q_j+\langle q_jk\rangle\right)}\phi_k.
\]

\end{theo}

\begin{proof}
This is easily checked.
\end{proof}

\begin{flushleft}
\textbf{General type case:}\end{flushleft} In this case, the Landau-Ginzburg point is a regular singular point of the Picard-Fuchs operator. After the change of variables $q=t^{-d}$, the Picard-Fuchs equation becomes
\begin{equation}\label{picard-fuchs-general-type}
\left[t^d\prod_{j=1}^N\prod_{c=0}^{w_j-1}(-q_jD_t-c)-\prod_{c=1}^d(-D_t+c)\right]I(t)=0\quad\text{where $D_t:=t\frac{d}{dt}$.}
\end{equation}
We have the following result:
\begin{theo}\label{solutions-general-type}
A complete set of solutions to the irreducible component of Equation $(\ref{picard-fuchs-general-type})$ is given by
\[
I_{\text{FJRW}}^{small}(t,-1)=\sum_{k\in \textbf{Nar}} \sum_{l=0}^{\infty}\frac{(-t)^{dl+k+1}}{(dl+k)!}\prod_{j=1}^N \frac{(-1)^{\lfloor q_j(dl+k)\rfloor}\Gamma\left(q_j(dl+k+1)\right)}{\Gamma \left(q_j+\langle q_jk\rangle\right)}\phi_k.
\]
\end{theo}
\begin{proof}
This is not hard to check.
\end{proof}
\begin{cor}
In the general type case, the monodromy matrix at the Landau-Ginzburg point is diagonal.
\end{cor}

\subsection{Matching the $I$-functions via asymptotic expansion} In this section we will find asymptotic expansions for the small $I$-functions. Theorems \ref{fano-genus-zero-correspondence} and \ref{general-type-genus-zero-correspondence} will follow from this.

\subsubsection{Summing formal power series: Borel summation}

The main tool we will use in the proofs of Theorems \ref{fano-genus-zero-correspondence} and \ref{general-type-genus-zero-correspondence} is known as \textit{Borel summation}. Given a divergent power series which is a formal solution of a differential equation at an irregular singular point, Borel summation allows us find some analytic function in a sector of the complex plane, which is a solution to the differential equation, and that has as asymptotic expansion at the irregular point the original divergent series. The idea behind this type of summation is to regularized a divergent series by dividing its coefficients by a Gamma function. This regularized series will be convergent in a disk of finite radius. We may apply a Laplace transformation provided that the regularized series can be analytically continued to infinity along some suitable ray. This procedure will give a new analytic function which is a solution to the original equation. Invoking Watson's lemma, we can show that the asymptotic expansion at infinity of the new function is given by the original divergent series. A simple example of this type of construction can be found in \cite[pp. 246-251]{Miller}.

\subsubsection{The regularized FJRW I-function for the Fano case}
Define the \textit{regularized} FJRW $I$-function as
\begin{equation}\label{regularized-I-function}
I_{\text{FJRW}}^{reg}(\tau):=\sum_{k\in \textbf{Nar}} \sum_{l=0}^{\infty} \frac{\tau^{r(l+\frac{k+1}{d})}}{(dl+k)!}\frac{(-1)^{dl+k+1}}{\Gamma\left(1+r\frac{dl+k+1}{d}\right)}\prod_{j=1}^N \frac{(-1)^{\lfloor q_j(dl+k)\rfloor}\Gamma\left(q_j(dl+k+1)\right)}{\Gamma \left(q_j+\langle q_jk\rangle\right)}\phi_k.
\end{equation}
This series converges absolutely for $|\tau|^r< r^r d^{d}\prod_{j=1}^Nw_j^{-w_j}$, and it is a solution of the \textit{regularized Picard-Fuchs equation}
\begin{equation}\label{regularized Picard-Fuchs equation}
\left[\prod_{j=1}^N\prod_{c=0}^{w_j-1}\left(-\frac{w_j}{r}\tau\frac{d}{d\tau}-c\right)-\tau^{-r}\prod_{c=0}^{r-1}\left(\tau\frac{d}{d\tau}-c\right)\prod_{c=1}^d \left(-\frac{d}{r}\tau\frac{d}{d\tau}+c\right)\right] I_{\text{FJRW}}^{reg}(\tau)=0.
\end{equation}

\begin{lemma}\label{continuation}
The regularized FJRW $I$-function $I_{\text{FJRW}}^{reg}(\tau) $ can be analytically continued to $\tau=\infty$. 
\end{lemma}
\begin{proof}
It is not hard to see that the regularized Picard-Fuchs equation has singularities at $\tau=0,\infty$ and for $\tau$ satisfying $(-\tfrac{\tau}{r})^r=d^{d}\prod_{j=1}^Nw_j^{-w_j}$, and one can check that all these singular points are regular. It follows that $I_{\text{FJRW}}^{reg}$ can be analytically continued to $\tau=\infty$ along any ray that avoids these singularities. 
\end{proof}


\subsubsection{Proof of Theorem \ref{fano-genus-zero-correspondence} (Genus zero LG/Fano Correspondence)}

To prove Theorem \ref{fano-genus-zero-correspondence} we need to construct a holomorphic function whose asymptotic expansion is given by the small FJRW $I$-function. Define this function to be a Laplace integral of the regularized FJRW $I$-function:
\begin{equation}
\mathbb{I}_{\text{FJRW}}(u):=u\mathcal{L}(I_{\text{FJRW}}^{reg})(u)=u\int_{0}^{\infty}e^{-u\tau}I_{\text{FJRW}}^{reg}(\tau)d\tau,
\end{equation}
where the ray of integration is any ray that avoids the singular points of Equation (\ref{regularized Picard-Fuchs equation}). It follows that this function is holomorphic for $|\arg(u)|<\pi/r$. As a consequence of Watson's lemma, we have the following relation 
\begin{equation}\label{new_I_asymptotic_expansion}
\mathbb{I}_{\text{FJRW}}(u)\sim \sum_{k\in \textbf{Nar}} \sum_{l=0}^{\infty} \frac{(-1)^{dl+k+1}}{u^{r(l+\frac{k+1}{d})}(dl+k)!}\prod_{j=1}^N \frac{(-1)^{\lfloor q_j(dl+k)\rfloor}\Gamma\left(q_j(dl+k+1)\right)}{\Gamma \left(q_j+\langle q_jk\rangle\right)}\phi_k,
\end{equation}
as $u\rightarrow\infty$ from the region $|\arg(u)|<\pi/r$.

\begin{lemma}\label{new_I_Picard}
The holomorphic function $\mathbb{I}_{\text{FJRW}}$ satisfies the following differential equation:
\[
\left[\prod_{j=1}^N\prod_{c=0}^{w_j-1}\left(\frac{w_j}{r}u\frac{d}{du}-c\right)-u^r\prod_{c=1}^d \left(\frac{d}{r}u\frac{d}{du}+c\right)\right] \mathbb{I}_{\text{FJRW}}(u)=0.
\]
\end{lemma}
\begin{proof}
Let $f(\tau)$ be holomorphic in some region of the complex plane containing a ray on which we can define the Laplace transform of $f$. Moreover, let $f(0)=0$ and define $F(u):=u\int_{0}^{\infty} e^{-u\tau}f(\tau)d\tau$. Using the properties of the Laplace transform is easily seen that 
\begin{equation}\label{Laplace_property}
\begin{split}
u\mathcal{L}\left(\alpha \tau\frac{d}{d\tau}f(\tau)+\beta f(\tau)\right)&=-\alpha u \frac{d}{du} \left(\mathcal{L}\left(\frac{d}{d\tau}f(\tau)\right)\right)+\beta u\mathcal{L}\left(f(\tau)\right)\\
&=-\alpha u \frac{d}{du} \left(u\mathcal{L}(f(\tau))-f(0)\right)+\beta u\mathcal{L}\left(f(\tau)\right)\\
&=\left(-\alpha u \frac{d}{du} +\beta\right)F(u),
\end{split}
\end{equation}
where $\alpha$ and $\beta$ are arbitrary complex numbers.

Now, Equation (\ref{regularized Picard-Fuchs equation}) can be rewritten as
\[
\left[\prod_{j=1}^N\prod_{c=0}^{w_j-1}\left(-\frac{w_j}{r}\tau\frac{d}{d\tau}-c\right)-\left(\frac{d}{d\tau}\right)^r\prod_{c=1}^d \left(-\frac{d}{r}\tau\frac{d}{d\tau}+c\right)\right] I_{\text{FJRW}}^{reg}(\tau)=0.
\]
Applying Equation (\ref{Laplace_property}) iteratively to this yields the desired result. 
\end{proof}
\begin{cor}\label{new_I_Picard_cor}
The holomorphic function $\mathbb{I}_{\text{FJRW}}(u=q^{1/r})$ satisfies the Picard-Fuchs equation of $\mathcal{X}_W$, i.e.
\[
\left[\prod_{j=1}^N\prod_{c=0}^{w_j-1}\left(w_jq\frac{d}{dq}-c\right)-q\prod_{c=1}^d \left(dq\frac{d}{dq}+c\right)\right] \mathbb{I}_{\text{FJRW}}(u=q^{1/r})=0.
\]
\end{cor}
\begin{proof}
This easily follows from Lemma \ref{new_I_Picard} after the change of variables $q=u^r$.
\end{proof}
We now have all the ingredients we need for the proof of Theorem \ref{fano-genus-zero-correspondence}. Corollary \ref{new_I_Picard_cor} implies that  $\mathbb{I}_{\text{FJRW}}(u=q^{1/r})$ is  a holomorphic solution to the Picard-Fuchs equation of $\mathcal{X}_W$ at $q=0$. Since $I_{\text{GW}}^{small}(q,1)$ is a complete set of solutions of the irreducible component of the Picard-Fuchs equation at $q=0$, there must exist a unique linear operator $L_{\text{GW}}:H_{\text{CR}}^{amb}(\mathcal{X}_W;\C)\longrightarrow H_{\text{FJRW}}^{nar}\left(W\right)$ such that $L_{\text{GW}}\cdot I_{GW}^{small}(q,1)=\mathbb{I}_{\text{FJRW}}(u=q^{1/r})$. It now follows from Equation (\ref{new_I_asymptotic_expansion}) that
\[
\begin{split}
L_{\text{GW}}\cdot I_{GW}^{small}(q,1)&=\mathbb{I}_{\text{FJRW}}(u=q^{1/r})\\
&\sim \sum_{k\in \textbf{Nar}} \sum_{l=0}^{\infty} \frac{(-1)^{dl+k+1}}{q^{(l+\frac{k+1}{d})}(dl+k)!}\prod_{j=1}^N \frac{(-1)^{\lfloor q_j(dl+k)\rfloor}\Gamma\left(q_j(dl+k+1)\right)}{\Gamma \left(q_j+\langle q_jk\rangle\right)}\phi_k \\ 
&=I_{\text{FJRW}}^{small}(t=q^{-1/d},-1),
\end{split}
\]
as $q\rightarrow\infty$ from some suitable sector. This completes the proof of the theorem.
\begin{flushright}
$\square$
\end{flushright}
\begin{cor}[Corollary \ref{corollary1.4}]\label{invariants_corollary}
In the Fano case $(i.e.\text{ } \K<0)$, the genus zero FJRW invariants are completely determined by the genus zero Gromov-Witten invariants of $\mathcal{X}_W$.
\end{cor}
\begin{proof}
The genus zero Gromov-Witten invariants of  $\mathcal{X}_W$ completely determine $I_{\text{GW}}^{small}(q,1)$. It follows from Theorem \ref{fano-genus-zero-correspondence} that $I_{\text{GW}}^{small}(q,1)$ completely determines $I_{\text{FJRW}}^{small}(t,-1)$. Using the procedure outlined in Section \ref{section_z=1} we can recover $I_{\text{FJRW}}^{small}(t,z)$. By using Equation (\ref{big-small}) we can obtain the big $I$-function $I_{\text{FJRW}}(\textbf{t},z)$ from the small FJRW $I$-function. The big $I$-function completely determines the FJRW $J$-function by means of the Mirror Theorem. The $J$-function completely determines the FJRW Lagrangian cone and, therefore, all genus zero FJRW invariants.
\end{proof}

\subsubsection{The regularized $GW$ $I$-function for the general type case} 
Define the \textit{regularized} GW $I$-function to be
\begin{equation}\label{regularized-GW-I-function}
I^{reg}_{\text{GW}}(\tau):=\sum_{f\in F}\mathbf{1}_f \sum_{n=0}^{\infty} \frac{\tau^{\K(n+\overline{f}+P)}}{\Gamma(1+\K(n+\overline{f}+P))}
\frac{\Gamma(1+d(n+\overline{f}+P))}{\Gamma(1+d(\overline{f}+P))}
\prod_{j=1}^N\frac{\Gamma(1+w_j(\overline{f}+P))}{\Gamma(1+w_j(n+\overline{f}+P))},
\end{equation}
where $\overline{f}:=\langle 1-f\rangle$, for $f\in F$.
This series converges absolutely for $|\tau|^{\K}< \K^{\K} d^{-d}\prod_{j=1}^Nw_j^{w_j}$, and it is a solution of the \textit{regularized Picard-Fuchs equation}
\begin{equation}\label{regularized Picard-Fuchs equation gt}
\left[\prod_{c=0}^{\K-1}\left(\tau\frac{d}{d\tau}-c\right) \prod_{j=1}^N\prod_{c=0}^{w_j-1}\left(\frac{w_j}{\K}\tau\frac{d}{d\tau}-c\right)-\tau^{\K}\prod_{c=1}^d \left(\frac{d}{\K}\tau\frac{d}{d\tau}+c\right)\right] I^{reg}_{\text{GW}}(\tau)=0.
\end{equation}
\begin{lemma}\label{GW_continuation}
The regularized GW $I$-function $I_{\text{GW}}^{reg}(\tau) $ can be analytically continued to $\tau=\infty$. 
\end{lemma}
\begin{proof}
It is not hard to see that the regularized Picard-Fuchs equation has singularities at $\tau=0,\infty$ and for $\tau$ satisfying $(\tfrac{\tau}{\K})^{\K}=d^{-d}\prod_{j=1}^Nw_j^{w_j}$, and one can check that all these singular points are regular. It follows that $I_{\text{GW}}^{reg}$ can be analytically continued to $\tau=\infty$ along any ray that avoids these singularities. 
\end{proof}

\subsubsection{Proof of Theorem \ref{general-type-genus-zero-correspondence} (Genus zero LG/General type Correspondence)}
The proof of Theorem \ref{general-type-genus-zero-correspondence} is similar to the proof of the LG/Fano correspondence. We need to construct a holomorphic function whose asymptotic expansion is given by the small GW $I$-function. Define this function to be a Laplace integral of the regularized GW $I$-function:
\begin{equation}
\mathbb{I}_{\text{GW}}(u):=u\mathcal{L}(I_{\text{GW}}^{reg})(u)=u\int_{0}^{\infty}e^{-u\tau}I_{\text{GW}}^{reg}(\tau)d\tau,
\end{equation}
where the ray of integration is any ray that avoids the singular points of Equation (\ref{regularized Picard-Fuchs equation gt}). It follows that this function is holomorphic for $|\arg(u)|<\text{min}(\pi/\K,\pi/2)$. As a consequence of Watson's lemma, we have the following relation 
\begin{equation}\label{new_I_asymptotic_expansion_GW}
\mathbb{I}_{\text{GW}}(u)\sim \sum_{f\in F}\mathbf{1}_f \sum_{n=0}^{\infty} \frac{1}{u^{\K(n+\overline{f}+P)}}
\frac{\Gamma(1+d(n+\overline{f}+P))}{\Gamma(1+d(\overline{f}+P))}
\prod_{j=1}^N\frac{\Gamma(1+w_j(\overline{f}+P))}{\Gamma(1+w_j(n+\overline{f}+P))},
\end{equation}
as $u\rightarrow\infty$ from the region $|\arg(u)|<\text{min}(\pi/\K,\pi/2)$.

\begin{lemma}\label{new_I_Picard_GW}
The holomorphic function $\mathbb{I}_{\text{GW}}$ satisfies the following differential equation:
\[
\left[u^{\K}\prod_{j=1}^N\prod_{c=0}^{w_j-1}\left(-\frac{w_j}{\K}u\frac{d}{du}-c\right)-\prod_{c=1}^d \left(-\frac{d}{\K}u\frac{d}{du}+c\right)\right] \mathbb{I}_{\text{GW}}(u)=0.
\]
\end{lemma}
\begin{proof}
The proof is identical to that of Lemma \ref{new_I_Picard}.
\end{proof}
\begin{cor}\label{new_I_Picard_cor_GW}
The holomorphic function $\mathbb{I}_{\text{GW}}(u=t^{d/{\K}})$ satisfies the following Picard-Fuchs equation:
\[
\left[t^{d}\prod_{j=1}^N\prod_{c=0}^{w_j-1}\left(-\frac{w_j}{d}t\frac{d}{dt}-c\right)-\prod_{c=1}^d \left(-t\frac{d}{dt}+c\right)\right] \mathbb{I}_{\text{GW}}(u=t^{d/\K})=0.
\]
\end{cor}
\begin{proof}
This easily follows from Lemma \ref{new_I_Picard_GW} after the change of variables $t^d=u^{\K}$.
\end{proof}
We now have all the ingredients we need for the proof of Theorem \ref{general-type-genus-zero-correspondence}. Corollary \ref{new_I_Picard_cor_GW} implies that  $\mathbb{I}_{\text{GW}}(u=t^{d/\K})$ is  a holomorphic solution to the Picard-Fuchs Equation (\ref{picard-fuchs-general-type}) at $t=0$. Since $I^{small}_{\text{FJRW}}(t,-1)$ is a complete set of solutions of the irreducible component of this Picard-Fuchs equation at $t=0$, there must exist a unique linear operator $L_{\text{FJRW}}:H_{\text{FJRW}}^{nar}\left(W\right)\longrightarrow H_{\text{CR}}^{amb}(\mathcal{X}_W;\C)$ such that $L_{\text{FJRW}}\cdot I_{FJRW}^{small}(t,-1)=\mathbb{I}_{\text{GW}}(u=t^{d/\K})$. It now follows from Equation (\ref{new_I_asymptotic_expansion_GW}) that
\[
\begin{split}
L_{\text{FJRW}}\cdot I_{FJRW}^{small}(t,-1)&=\mathbb{I}_{\text{GW}}(u=t^{d/\K})\\
&\sim \sum_{f\in F}\mathbf{1}_f \sum_{n=0}^{\infty} \frac{1}{t^{d(n+\overline{f}+P)}}
\frac{\Gamma(1+d(n+\overline{f}+P))}{\Gamma(1+d(\overline{f}+P))}
\prod_{j=1}^N\frac{\Gamma(1+w_j(\overline{f}+P))}{\Gamma(1+w_j(n+\overline{f}+P))} \\ 
&=I_{\text{GW}}^{small}(q=t^{-d},1)
\end{split}
\]
as $t\rightarrow\infty$ from some suitable sector. This completes the proof of the theorem.
\begin{flushright}
$\square$
\end{flushright}
\begin{cor}[Corollary \ref{corollary1.6}]
If $\K>0$, the genus zero GW invariants of $\mathcal{X}_W$ are completely determined by the genus zero FJRW invariants of the pair $\left(W,\langle J_W\rangle\right)$.
\end{cor}
\begin{proof}
The proof is similar to that of Corollary \ref{invariants_corollary}.
\end{proof}

%

\subsection{Massive vacuum solutions in the Fano case}

The formal solutions described in Theorem \ref{formal-solutions} only represent a subset of the set of solutions to the Picard-Fuchs equation at the Landau-Ginzburg point. This occurs because $I_{\text{FJRW}}^{small}$ has smaller rank than $I_{\text{GW}}^{small}$ as a cohomology-valued functions. In addition to the solutions represented by the small FJRW $I$-function,  the number of solutions needed to obtain a complete set is equal to the Fano index of $\mathcal{X}_W$, $r:=-\K=\sum_j w_j-d$. An effort to find the remaining solutions has led us to the following $I$-function:
\begin{equation}\label{mass-I-function}
I_{j,mass}(q):=q^{-\frac{N-2}{2r}} e^{\alpha_j q^{\frac{1}{r}} }\sum_{n=0}^{\infty}\frac{a_{j,n}}{q^{n/r}},
\end{equation}
where  $\alpha_j$ is one the the $r$ solutions of the equation $\left(\frac{\alpha}{r}\right)^r=d^d\prod_{j=1}^N w_j^{-w_j}$, and the coefficients $\{a_{j,n}\}$ can be obtained recursively from Equation (\ref{picard-fuchs-fano}). In the physics literature, the reduction in the dimension of the state space was due to the appearance of certain "massive vacuum" solutions. We believe $I_{j, mass}(q)$ plays the role of a quantum contribution due to these massive vacuum solutions. \begin{theo}\label{massive-solutions}
The functions $I_{j,mass}(q)$ defined in Equation  $(\ref{mass-I-function})$ are formal solutions of the Picard-Fuchs equation $(\ref{picard-fuchs-fano})$ at the Landau-Ginzburg point $q=\infty$. Together with $I_{\text{FJRW}}^{small}$ they form a complete set of solutions.
\end{theo}
\begin{proof}
We begin by making the following change of variables: $q=u^r$. Then, Equation (\ref{picard-fuchs-fano}) becomes 
\begin{equation}\label{picard-fuchs-change}
\left[\prod_{j=1}^N\prod_{c=0}^{w_j-1}\left(\frac{w_j}{r}D_u-c\right)-u^r\prod_{c=1}^d\left(\frac{d}{r}D_u+c\right)\right]I(u)=0.
\end{equation} 
We are looking for solutions of the form $I(u)=e^{\alpha u}\sum_{n=0}^{\infty}a_{\alpha,n} u^{-\lambda-n}$, with $\alpha\neq 0$. After plugging this solution into Equation (\ref{picard-fuchs-change}), we obtain a relation for the highest power of $u$ given by
\begin{equation}\label{highest-power}
\left(\frac{\alpha}{r}\right)^{\sum_j w_j}\prod_{j=1}^Nw_j^{w_j}u^{\sum_j w_j}e^{\alpha u}\frac{a_{\alpha,0}}{u^{\lambda}}-u^r\left(\frac{\alpha}{r}\right)^{d}d^d u^d e^{\alpha u}\frac{a_{\alpha,0}}{u^{\lambda}}=0,
\end{equation}
which implies that
\begin{equation}\label{roots}
\left(\frac{\alpha}{r}\right)^r=d^d\prod_{j=1}^N w_j^{-w_j}.
\end{equation}
The relation for the second highest power of $u$ is given by
\[
\begin{split}
&\frac{1}{2}\left(\sum_{j=1}^Nw_j\right)\left(-1+\sum_{j=1}^Nw_j\right)\frac{\alpha^{-1+\sum_{j}w_j}}{r^{\sum_j w_j}}\prod_{j=1}^Nw_j^{w_j}a_{\alpha,0}-\left(\sum_{j=1}^N w_j\right)\frac{\alpha^{-1+\sum_j w_j}}{r^{\sum_j w_j}}\prod_{j=1}^Nw_j^{w_j}a_{\alpha,0}\lambda\\
&-\sum_{i=1}^N\left(\prod_{j\neq i}w_j^{w_j}\right) w_i^{w_i-1}\frac{(w_i-1)w_i}{2}\frac{\alpha^{-1+\sum_j w_j}}{r^{-1+\sum_j w_j}}a_{\alpha,0}-\frac{d(d-1)}{2}\alpha^{d-1}\left(\frac{d}{r}\right)^da_{\alpha,0}\\
&+d\alpha^{d-1}\left(\frac{d}{r}\right)^da_{\alpha,0}\lambda
-\frac{d(d+1)}{2}\alpha^{d-1}\left(\frac{d}{r}\right)^{d-1}a_{\alpha,0}
+\left(\left(\frac{\alpha}{r}\right)^{\sum_j w_j}\prod_{j=1}^Nw_j^{w_j} -\left(\frac{\alpha}{r}\right)^{d}d^d\right) a_{\alpha,1}=0
\end{split}
\]
The last two terms of this relation vanish because of Equation (\ref{highest-power}). Thus, we are left with
\[
\begin{split}
&\frac{1}{2}\left(\sum_{j=1}^Nw_j\right)\left(-1+\sum_{j=1}^Nw_j\right)\frac{\alpha^{-1+\sum_{j}w_j}}{r^{\sum_j w_j}}\prod_{j=1}^Nw_j^{w_j}a_{\alpha,0}-\left(\sum_{j=1}^N w_j\right)\frac{\alpha^{-1+\sum_j w_j}}{r^{\sum_j w_j}}\prod_{j=1}^Nw_j^{w_j}a_{\alpha,0}\lambda\\
&-\sum_{i=1}^N\left(\prod_{j\neq i}w_j^{w_j}\right) w_i^{w_i-1}\frac{(w_i-1)w_i}{2}\frac{\alpha^{-1+\sum_j w_j}}{r^{-1+\sum_j w_j}}a_{\alpha,0}-\frac{d(d-1)}{2}\alpha^{d-1}\left(\frac{d}{r}\right)^da_{\alpha,0}\\
&+d\alpha^{d-1}\left(\frac{d}{r}\right)^da_{\alpha,0}\lambda
-\frac{d(d+1)}{2}\alpha^{d-1}\left(\frac{d}{r}\right)^{d-1}a_{\alpha,0}=0.
\end{split}
\]

Multiplying by $\alpha$ and using Equation (\ref{roots}) yields
\[
\begin{split}
&\frac{1}{2}\left(\sum_{j=1}^Nw_j\right)\left(-1+\sum_{j=1}^Nw_j\right)a_{\alpha,0}-\left(\sum_{j=1}^N w_j\right)a_{\alpha,0}\lambda-r\sum_{i=1}^N\frac{w_i-1}{2}a_{\alpha,0}-\frac{d(d-1)}{2}a_{\alpha,0}\\
&+da_{\alpha,0}\lambda-r\frac{(d+1)}{2}a_{\alpha,0}=0,
\end{split}
\]
which implies
\[
\frac{r}{2}\left(d+\sum_{j=1}^Nw_j\right)-\frac{r}{2}
-\frac{r}{2}\left( -N+\sum_{i=1}^N w_i\right)-r\lambda-r\frac{(d+1)}{2}=0,
\]
from which it follows that 
\[
\lambda=\frac{N-2}{2}.
\]
Thus, Equation (\ref{picard-fuchs-change}) has solutions of the form 
\[
u^{\frac{N-2}{2}}e^{\alpha u}\sum_{n=0}^{\infty} \frac{a_{\alpha ,n}}{u^n},
\]
where a recursion for $\{a_{\alpha, n}\}$ can be obtained from the differential equation, and $\alpha$ satisfies Equation (\ref{roots}). The statement of the theorem follows after making the change of variables $u=q^{1/r}$.
\end{proof}
\begin{cor}
Using the basis described in Theorems $\ref{formal-solutions}$ and $\ref{massive-solutions}$, the formal monodromy matrix at the Landau-Ginzburg point $q=\infty$ is given by 

\begin{equation}
 \left(
    \begin{array}{r@{}c|c@{}l}
  &    \begin{smallmatrix}
  \mbox{\huge D}\rule[-1ex]{0pt}{2ex}

      \end{smallmatrix} & \mbox{\huge0} \\\hline
  &    \mbox{\huge0} &  
       \begin{smallmatrix}\rule{0pt}{2ex}
        0 & & & \cdots&\cdots & & 0 & \lambda_r\\
        \lambda_1 & 0 &  & &\cdots & \cdots& & 0\\
        0 & \lambda_2 & 0 \\
         & 0 & \lambda_3 & & & & &\\
        \vdots&  &  & &\ddots& & & \vdots\\
        & &\\
        & & &&& \lambda_{r-2}& 0\\
        0 & & &   \cdots & & 0 & \lambda_{r-1}& 0
      \end{smallmatrix}    
    \end{array} 
\right)
\end{equation}

where $D$ is diagonal and $\lambda_j\neq0$ for $j=1,\dots,r$. It follows that the monodromy matrix is diagonalizable. 
\end{cor}
\begin{proof}
This follows from the fact that the asymptotic behavior of $I_{j,mass}(q)$ is given by $q^{-\frac{N-2}{2r}} e^{\alpha_j q^{\frac{1}{r}} }$ by Theorem \ref{massive-solutions}.
\end{proof}

\subsubsection{Relating the GW $I$-function to the massive vacuum solutions in the Fano case}
We will now study the asymptotics of the Gromov-Witten $I$-function of a degree $d$ hypersurface in projective space $\mathbb{P}^{N-1}$. We will compute the asymptotic behavior under an asymptotic sequence as $q\rightarrow\infty$ of the form $\left\{ \frac{e^{\alpha q^{\frac{1}{r}}}}{q^{\frac{\lambda}{r}}}, \frac{e^{\alpha q^{\frac{1}{r}}}} {q^{\frac{1+\lambda}{r}} },\frac{e^{\alpha q^{\frac{1}{r}}}} {q^{\frac{2+\lambda}{r}}} ,\dots \right\}$. The following result shows that the asymptotic expansion of the GW $I$-function will match one of the vacuum solutions to the Picard-Fuchs equation:

\begin{theo}[Theorem \ref{intro-massive-theorem}]\label{asymptotic-massive}
Let $I_{\text{GW}}(q,1)$ be the Gromov-Witten $I$-function of a degree $d$ hypersurface inside $\mathbb{P}^{N-1}$. Then,
\[
I_{\text{GW}}(q,1)\sim C' \frac{\Gamma(1+P)^N}{\Gamma(1+dP)} q^{-\frac{(N-2)}{2r}} e^{\alpha q^{\frac{1}{r}}} (1+\mathcal{O}\left(q^{-\frac{1}{r} } )\right)\quad\text{as $q\rightarrow +\infty$ along the real axis,}
\]
where $C'$ is a constant, $r=N-d$, and $\alpha>0$ satisfies $\left(\frac{\alpha}{r}\right)^r=d^d$.
\end{theo}
\begin{proof}
Let $i:\mathcal{X}_W\rightarrow \mathbb{P}^{N-1}$ be the embedding of the degree $d$ hypersurface $\mathcal{X}_{W}=\{x_1^d+\dots+x_N^d=0\}$ into projective space. The Gromov-Witten $I$-function for $\mathbb{P}^{N-1}$ is given by
\[
I_{\text{GW},\mathbb{P}}(\tau,1)=\sum_{n=0}^{\infty}\frac{\tau^{n+H}}{\prod_{m=1}^n(P+m)^N},
\]
where $H$ is the hyperplane class in $\mathbb{P}^{N-1}$. The asymptotic expansion of this function along the positive real axis was computed in \cite[Section 6]{Iritani}:
\[
I_{\text{GW},\mathbb{P}}(\tau)\sim C\cdot\Gamma(1+H)^N\tau^{-\frac{N-1}{2N}}e^{N\tau^{\frac{1}{N}}}(1+\mathcal{O}(\tau^{-1/N}))\quad\text{as $\tau\rightarrow\infty$ along the positive real axis,}
\]
where $C$ is a constant.
To compute the asymptotic expansion of the $I$-function of $\mathcal{X}_W$, we use the quantum Lefschetz principle to write the $I$-function as a Laplace-type integral:
\begin{equation}\label{asymptotics}
I_{\text{GW}}(q=u^d,1)=\sum_{n=0}^{\infty} \frac{\prod_{m=1}^{dn} (dP+m)}{\prod_{m=1}^n(P+m)^N}u^{dn+dP}=\frac{1}{\Gamma(1+dP)u}\int_{0}^{\infty} i^{\ast}I_{\mathbb{P}}(\tau^d)e^{\tau/u}d\tau,
\end{equation}
where $P:=i^{\ast}H$.
We can now apply the method of steepest descent to this integral. The leading behavior in the asymptotic expansion will be given by
\[
C\cdot\frac{\Gamma(1+P)^N}{\Gamma(1+dP)u}\int_{\epsilon}^{\infty} \tau^{-\frac{d(N-1)}{2N}}e^{N\tau^{\frac{d}{N}}}e^{\tau/u}d\tau,
\]
where $\epsilon$ is a sufficiently small positive number. After the change of variables $\tau=u^{N/{(N-d)}}t$, the integral becomes
\[
C\cdot\frac{\Gamma(1+P)^N}{\Gamma(1+dP)} u^{-\frac{d(N-3)}{2(N-d)}}\int_{\epsilon}^{\infty} t^{-\frac{d(N-1)}{2N}}e^{u^{d/(N-d)} (Nt^{d/N}-t)}dt.
\]
if we let $\lambda=u^{d/(N-d)}$, the integral now becomes
\[
C\cdot\frac{\Gamma(1+P)^N}{\Gamma(1+dP)} \lambda^{-\frac{(N-3)}{2}}\int_{\epsilon}^{\infty} t^{-\frac{d(N-1)}{2N}}e^{\lambda (Nt^{d/N}-t)}dt.
\]
To find the leading order in the expansion, we need to find the maxima of the function $f(t):=Nt^{d/N}-t$. This function has a critical point at $t_0=d^{N/{(N-d)}}$. Since $d<N$, we have that $f''(t_0)=d\left(\frac{d-N}{N}\right)t_0^{\frac{d}{N}-2}<0$, and therefore, $t_0$ is a true maximum of $f(t)$. This implies that the asymptotic behavior is given by (see, for example, \cite[Equation (3.17)]{Miller}):
\[
C\sqrt{\frac{2\pi Nd^{\frac{2N-d}{N-d}}}{d(N-d)}}d^{-\frac{d(N-1)}{2(N-d)}}\frac{\Gamma(1+P)^N}{\Gamma(1+dP)} \lambda^{-\frac{1}{2}}\lambda^{-\frac{(N-3)}{2}} e^{\lambda d^{d/(N-d)}(N-d)}(1+\mathcal{O}\left(\lambda^{-1})\right).
\]
Therefore, the leading asymptotic behavior of $I_{\text{GW}}(q,1)$ is given by
\[
I_{\text{GW}}(q,1)\sim C' \frac{\Gamma(1+P)^N}{\Gamma(1+dP)} q^{-\frac{(N-2)}{2r}} e^{\alpha q^{\frac{1}{r}}} (1+\mathcal{O}\left(q^{-\frac{1}{r} } )\right)\quad\text{as $q\rightarrow +\infty$ along the real axis,}
\]
where $r=N-d$, and $\alpha>0$ satisfies $\left(\frac{\alpha}{r}\right)^r=d^d$. This completes the proof of the theorem.
\end{proof}
\begin{remark}
After the change of variables $q=u^r$, Theorem \ref{asymptotic-massive} states that 
\[
I_{\text{GW}}(u^r,1)\sim C' \frac{\Gamma(1+P)^N}{\Gamma(1+dP)} u^{-\frac{(N-2)}{2}} e^{\alpha u} (1+\mathcal{O}\left(u^{-1} )\right)\quad\text{as $u\rightarrow\infty$,}
\]
from the region $|\arg(u)|<\pi/2$. This means that after the change of variables $u\rightarrow e^{-\frac{2\pi i k}{r}}u$, for $k=0,\dots,r-1$, we obtain the expansion 
\[
I_{\text{GW}}(u^r,1)\sim C'e^{\frac{2\pi ik(N-2)}{2r}} \frac{\Gamma(1+P)^N}{\Gamma(1+dP)} u^{-\frac{(N-2)}{2}} e^{\alpha e^{-\frac{2\pi i k}{r}}u} (1+\mathcal{O}\left(u^{-1} )\right)\quad\text{as $u\rightarrow\infty$,}
\]
in the region $2\pi k/r -\pi/2<\arg(u)<2\pi k/r +\pi/2$. This shows that all formal solutions described in Theorem \ref{massive-solutions} can be obtained as asymptotic expansions of $I_{\text{GW}}$ along different sectors of the complex plane.
\end{remark}
\begin{remark}
The main content of Theorem \ref{asymptotic-massive} can be found in \cite[Section 5.3]{Iritani}. There, the asymptotic expansion is related to the so called Gamma Conjectures of Galkin-Golyshev-Iritani-Dubrovin.
\end{remark}
\subsection{Example: Degree $3$ del Pezzo surface}
Let $W=x_1^3+x_2^3+x_3^3+x_4^3$, and define $\mathcal{X}_W:=\{W=0\}\subset \mathbb{P}^3$. Then,  $\mathcal{X}_{W}$ is a degree 3 del Pezzo surface with Fano index $r=-\K =1$. The irreducible component of the Picard-Fuchs equation corresponding to $\mathcal{X}_W$ is given by
\[
\left[\left(zq\frac{d}{dq}\right)^3-3q\left(3zq\frac{d}{dq}+2z\right)\left(3zq\frac{d}{dq}+z\right)\right]I(q,z)=0.
\]
A complete set of solutions to this equation around $q=0$ is given by the Gromov-Witten I-function: 
\[
I_{\text{GW}}(q,z):=zq^{P/z}\sum_{n=0}^{\infty}q^{n}\frac{\prod_{m=1}^{3n}(3P+mz)}{\prod_{m=1}^n(P+mz)^4},
\]
where $P^3=0$. The corresponding formal solutions at $q=\infty$ are given by
\[
I_{\text{FJRW}}(t=q^{-\frac{1}{3}},-z):=-z\sum_{l=0}^{\infty}\frac{(-z)^l}{q^{l+\frac{1}{3}}(3l)!}\frac{\Gamma\left(l+\frac{1}{3}\right)^4}{\Gamma\left(\frac{1}{3}\right)^4}\phi_0 +\sum_{l=0}^{\infty}\frac{(-z)^l}{q^{l+\frac{2}{3}}(3l+1)!}\frac{\Gamma\left(l+\frac{2}{3}\right)^4}{\Gamma\left(\frac{2}{3}\right)^4}\phi_1, \text{ and}
\]
\[
I_{mass}(q):=q^{-1}e^{3^3q} \sum_{n=0}^{\infty}\frac{a_n}{q^{n}},
\]
where the coefficients $\{a_n\}$ satisfy the recursion $a_1=\frac{7}{243}a_0$ and $3^6(n+2)a_{n+2}=(54n^2+162n+129)a_{n+1}-(n+1)^3a_n$, for $n\geq 0$.

By choosing an asymptotic sequence of the form $\left\{\frac{q^{\lambda}e^{bq}}{q^n}\right\}_{n=0}^{\infty}$ for the asymptotic expansion as $q\rightarrow\infty$, we obtain 
\[
I_{\text{GW}}(q,1)\sim C \frac{\Gamma(1+P)^4}{\Gamma(1+3P)}q^{-1}e^{3^3q}\left(1+\mathcal{O}(q^{-1})\right),\quad\text{as $q\rightarrow \infty$ along the positive real q-axis.}
\]
If we choose an asymptotic sequence of the form $\{q^{\lambda-n}\}_{n=0}^{\infty}$ for the asymptotic expansion as $q\rightarrow\infty$, we obtain
\[
L_{\text{GW}}\cdot  I_{\text{GW}}(q,1)\sim I_{\text{FJRW}}(t=q^{-\frac{1}{3}},-1),\quad\text{as $q\rightarrow\infty$.}
\]
Thus, by means of Equation (\ref{big-small}), the big $I$-function of FJRW theory is completely determined by the Gromov-Witten $I$-function. This implies that the genus zero FJRW invariants are completely determined by the genus zero Gromov-Witten invariants.
\begin{flushright}
$\square$
\end{flushright}



\singlespace

\end{document}